\documentclass[11pt]{amsart}
\usepackage[utf8]{inputenc}

\usepackage{amsthm,amssymb}
\usepackage{amsmath}
\usepackage[framemethod=tikz]{mdframed}

\usepackage{url}
\usepackage[colorlinks,linkcolor=blue,anchorcolor=blue,citecolor=blue,backref=page]{hyperref}
\usepackage{color}
\usepackage{graphics,epsfig}
\usepackage{graphicx}
\usepackage{float} 
\usepackage[english]{babel}
\usepackage{mathtools}
\usepackage{todonotes}
\usepackage{url}
\usepackage[colorlinks,linkcolor=blue,anchorcolor=blue,citecolor=blue,backref=page]{hyperref}
\usepackage{breqn}
\usepackage{enumitem}
\usepackage{comment}
\usepackage{longtable}
\usepackage{bm}

\usepackage[norefs,nocites]{refcheck}

\usepackage{mathrsfs}
\hypersetup{breaklinks=true}

\newtheorem{thm}{Theorem}
\newtheorem{lem}[thm]{Lemma}

\newtheorem{cor}[thm]{Corollary}

\newtheorem{rem}[thm]{Remark}

\numberwithin{equation}{section}
\numberwithin{thm}{section}
\numberwithin{table}{section}

\def\squareforqed{\hbox{\rlap{$\sqcap$}$\sqcup$}}
\def\qed{\ifmmode\squareforqed\else{\unskip\nobreak\hfil
\penalty50\hskip1em\null\nobreak\hfil\squareforqed
\parfillskip=0pt\finalhyphendemerits=0\endgraf}\fi}

\def \balpha{\bm{\alpha}}
\def \bbeta{\bm{\beta}}

\def\cG{{\mathcal G}}

\def\cI{{\mathcal I}}

\def\cK{{\mathcal K}}

\def\cS{{\mathcal S}}
\def\cT{{\mathcal T}}

\def\inv{\mathrm{inv}} 
\def\sqrt{\mathrm{rt}} 
\def\sq{\mathrm{sq}}

\def \F {{\mathbb F}}

\def \Z {{\mathbb Z}}

\def \Ki {\F_q(T)_\infty}

\def\\{\cr}
\def\({\left(}
\def\){\right)}

\def\mand{\qquad\mbox{and}\qquad}

\newcommand{\ov}{\overbar}
\newcommand{\overbar}[1]{\mkern 1.5mu\overline{\mkern-1.5mu#1\mkern-1.5mu}\mkern 1.5mu}

 \newcommand{\legendre}[2]{\genfrac{(}{)}{}{}{#1}{#2}}
 \newcommand{\Mod}[1]{\ (\mathrm{mod}\ #1)}

\makeatletter
\def\l@subsection{\@tocline{2}{0pt}{2.8pc}{1pc}{}}
\def\l@subsubsection{\@tocline{2}{0pt}{5pc}{7.5pc}{}}
\makeatother

\title[Bilinear forms in function fields]{ Bilinear forms with Kloosterman and Gauss sums in function fields}
 \author[C.~Bagshaw]{Christian Bagshaw}
 \address{School of Mathematics and Statistics, University of New South Wales.
 Sydney, NSW 2052, Australia}
 \email{c.bagshaw@unsw.edu.au}
\keywords{function field, bilinear form, Kloosterman sum, Gauss sum, energy bounds}
\subjclass[2010]{11T06, 11T23}

\begin{document}
\begin{abstract}
In recent years, there has been a lot of progress in obtaining non-trivial bounds for bilinear forms of Kloosterman sums in $\Z/m\Z$ for arbitrary integers $m$. These results have been motivated by a wide variety of applications, such as improved asymptotic formulas for moments of $L$-functions. However, there has been very little work done in this area in the setting of rational function fields over finite fields. We remedy this and provide a number of new non-trivial bounds for bilinear forms of Kloosterman and Gauss sums in this setting, based on new bounds on the number of solutions to certain modular congruences in $\F_q[T]$ .
\end{abstract}

\maketitle

\tableofcontents

\section{Introduction}
In recent years, much effort has been dedicated to studying certain bilinear forms of Kloosterman sums. That is, bounding sums of the form 
$$\mathcal{K}_m(\balpha, \bbeta; M,N)=\sum_{s=1}^M\sum_{t=1}^N\alpha_s\beta_t\sum_{\substack{x=1 \\ (x,m) = 1}}^me_m\left({sx + t\ov{x}} \right)$$
for complex weights $\balpha = \{\alpha_s\}$ and $\bbeta = \{\beta_t\}$ supported on $\{1,..., M\}$ and $\{1,...,N\}$ respectively, where $e_m(x) = \exp(2i\pi x/m)$ and where $\ov{x}$ denotes the inverse of $x$ modulo $m$. Not only is the study of these sums an interesting problem in its own right, but it has also been motivated by a variety of applications, including improved asymptotic formulas for moments of $L$-functions \cite{BlomerFouvryKowalskiMichelMilicevic2017_1, BlomerFouvryKowalskiMichelMilicevic2017_2, BlomerFouvryKowalskiMichel2017_2, KowalskiMichelSawin2017, Shparlinski2019_kloosterman, Wu20222}, the distribution of the divisor function in arithmetic progressions \cite{KerrShparlinski2020, LiuShparlinskiZhang2018divisor, WuXi2021, Xi2018}, as well as range of others \cite{BagShparlinski2023, KorolevShparlinski2020, KowalskiMichelSawin2018, LiuShparlinskiZhang2019, ShparlinskiZhang2016}. 

The Weil-Estermann bound \cite[Corollary 11.12]{IwaniecKowalski2004} implies that 
\begin{align*}
    \sum_{\substack{x=1 \\ (x,m) = 1}}^me_m\left({sx + t\ov{x}} \right) \ll \gcd(s,t,m)^{1/2}m^{1/2+o(1)}
\end{align*}
from which one can derive a trivial bound
$$\mathcal{K}_q(\balpha, \bbeta; M,N) \ll \|\balpha\|_1\|\bbeta\|_1q^{1/2 + o(1)}.$$
Thus, work in this area often revolves around improving upon this for as wide of a range on $S$ and $T$ as possible. Of particular interest is when $M,N \leq q^{1/2}$; the so-called ``P{\'o}lya-Vinogradov range". 

Here, similar to \cite{MacourtShparlinski2019}, we consider this problem in the setting of polynomials over finite fields. We model our approach after \cite{KerrShparlinskiWuXi2022, Shparlinski2019_kloosterman, ShparlinskiZhang2016}. In particular, we use the Cauchy-Schwarz and H{\"o}lder inequalities, as well as orthogonality relations, to reduce the problem to bounding counting functions corresponding to the number of solutions to  certain congruences over residue rings. We present our bounds on bilinear forms in terms of these counting functions, to hopefully highlight to what extent our methods could be improved. We also remark that our new bounds on these counting functions may be of independent interest (see Section \ref{sec:counting_solution}). As noted in \cite{Shparlinski2019_kloosterman, Shparlinski2022_Gauss}, similar techniques can also give rise to non-trivial bounds for bilinear forms of Gauss sums, which we also consider. 

We expect our bounds to have broad applicability to various arithmetic problems over $\F_q[T]$, particularly to moments of $L$-functions in function fields as seen in \cite{Florea2017, GaoZhao2023, Tamam2014}. In fact, by employing the techniques presented in \cite{BlomerFouvryKowalskiMichelMilicevic2017_2, Wu20222}, some of our results can potentially improve the error term of the main result in \cite{Tamam2014}. However, additional tools and reductions are necessary, such as bounding certain sums of divisor functions, which we plan to explore in forthcoming work.

Throughout this paper, we rely heavily on basic properties of Kloosterman and Gauss sums in $\F_q[T]$. While many of these properties are well-known and classical over the integers, some rely on newer ideas from \cite{KerrShparlinskiWuXi2022}. Although many of these properties are surely also well-known for $\F_q[T]$, we have been unable to locate proofs for most of the results needed. Therefore, we have included these proofs in Appendix~\ref{appendix:kloos_and_gauss}.

\section{Notation}
\subsection{General notation} A notation guide is given in Appendix \ref{notationguide} for ease of reference, but more complete explanations for the notation used are given here and throughout the paper. 

We fix a prime power $q$. Some of our results require $q$ to be odd, but in these cases this is explicitly stated. Let $\F_q$ be the finite field of order $q$ and let $\F_q[T]$ denote the ring of univariate polynomials with coefficients from $\F_q$. We also fix some $F \in \F_q[T]$ of degree $r$. 

For any $x \in \F_q[T]$ such that $\gcd(x,F) = 1$, we will denote by $\ov{x}$ the unique polynomial of degree less than $r$ such that $x\ov{x} \equiv 1 \Mod{F}$. We will assume this inverse is always taken to the modulus $F$, unless otherwise specified. Along similar lines, for any $x \in \F_q[T]$ there exists some unique $x' \in \F_q[T]$ such that $\deg x' < r$ and $x \equiv x' \Mod{F}$. Thus we will define $\deg_{F}x = \deg x'$. 

As in previous work regarding bilinear forms, we will attach to our variables some weights. That is, given some finite $\cS \subseteq \F_q[T]$, we will often associate to it some sequence of complex weights $\balpha = (\alpha_s)_{s \in \cS }$ and denote 
$$\|\balpha\|_\infty = \max_{s \in \cS }|\alpha_s| \mand  \|\balpha\|_\sigma = \bigg{(}\sum_{s \in \cS  }|\alpha_s|^\sigma\bigg{)}^{1/\sigma} \quad (\sigma\ge 1). $$
The most common sets we will work with are \textit{intervals}. In $\F_q[T]$, given some positive integer $k$ by an \textit{interval of size $q^k$}, we will mean a set of the form 
\begin{equation*}
    \begin{split}
\{x + x_0 : x \in \F_q[T], ~\deg x < k\}
    \end{split}
  \end{equation*}
for some $x_0 \in \F_q[T]$. We will call this interval \textit{initial} if $x_0 = 0$. Perhaps the most common type of interval that one would be interested in is the set of all monic polynomials of some given degree, but we consider them more generally.  Going forward, given positive integers $m$ and $n$, by $\cI_m$ and $\cI_n$ we will always mean intervals of size $q^m$ and $q^n$ respectively, and we will use the notation
\begin{equation}\label{eq:intervals}
    \begin{split}
&\cI_m = \{s + s_0 : s \in \F_q[T], ~\deg s < m\},\\  &\cI_n = \{t + t_0 : t \in \F_q[T], ~\deg t < m\}
    \end{split}
  \end{equation}
for some $s_0, t_0 \in \F_q[T]$. We have introduced notation for both of these sets to avoid having to define $s_0$ and $t_0$ each time. 

We will often write $(x,y)$ instead of $\gcd(x,y)$, especially when writing $\deg(x,y)$ instead of $\deg\gcd(x,y)$. Also, when summing over a variable $x \in \F_q[T]$ with the condition that $\deg x < m$, we will often not state the condition that $x \in \F_q[T]$. Furthermore, we will often write $x \equiv y (F)$ as opposed to $x \equiv y ~\Mod{F}$ to save space under a summation sign. But these will hopefully always be clear based on the context.  

\subsection{Modular congruences}\label{sec:notation_counting}
As mentioned previously, our new bounds for bilinear forms of Kloosterman and Gauss sums will largely be based on bounds (mostly new, and some existing) on the number of solutions to certain modular congruences in $\F_q[T]$, which can be found in Section~\ref{sec:counting_solution}. Here we will introduce some important notation pertaining to these congruences. 

Let $m$ and $n$ denote positive integers, $\cI_m$ and $\cI_n$ as in (\ref{eq:intervals}) and $a \in \F_q[T]$. We firstly define the two counting functions 
\begin{align}\label{eq:def:hyperbola}
H_{F,a}(\cI_m, \cI_n) = \# \{(x_1, x_2) \in \cI_m \times \cI_n :  x_1x_2 \equiv a \Mod{F}\}
\end{align}
and 
\begin{align}\label{eq:def:sum_inverse}
    I_{F,a}(\cI_m) = \# \{(x_1, x_2) \in \cI_m^2 : \ov{x}_1 + \ov{x}_2 \equiv a \Mod{F} \} .
\end{align}
We will also consider averages over $I_{F,a}(\cI_m)$, so for a positive integer $k$ we define 
\begin{align}\label{eq:def:sum_inverse_average}
    A_{F,a}(\cI_m, k) = \sum_{\deg h < k}I_{F,ah}(\cI_m).
\end{align}
If $\cI_m$ and $\cI_n$ are initial intervals we will write $H_{F,a}(m,n)$, $I_{F,a}(m)$ and $A_{F,a}(m, k)$ for simplicity. 

We next introduce three counting functions that could be considered measures of \textit{additive energy} of certain sets. 
\iffalse
In general, for a finite set $\cS \subseteq \F_q[T]$ we define 
\begin{align}\label{eq:def:energy}
        E_{F}&(\cS) \nonumber \\
    &= \#\{(x_1, x_2, x_3, x_4) \in \cS^4 : {x}_1 + {x}_2 \equiv {x}_3 + {x}_4 \Mod{F}\}, 
\end{align}
but we also consider some more specific variations.
\fi
We set 
\begin{align}\label{eq:def:energy_inv}
    E_{F}^{\inv}&(\cI_m) \nonumber \\
    &= \#\{(x_1, x_2, x_3, x_4) \in \cI_m^4 : \ov{x_1} + \ov{x_2} \equiv \ov{x_3} + \ov{x_4} \Mod{F}\}, 
\end{align}
\begin{align}\label{eq:def:energy_sq}
    E_{F}^{\sq}&(\cI_m) \nonumber \\
    &= \#\{(x_1, x_2, x_3, x_4) \in \cI_m^4 : {x}^2_1 + {x}^2_2 \equiv {x}^2_3 + {x}^2_4 \Mod{F}\}, 
\end{align}
and 
\begin{align}\label{eq:def:energy_sqrt}
    E_{F}^{\sqrt}&(m) \nonumber \\
    &= \#\{(x_1, x_2, x_3, x_4) \in \F_q[T]^4 : \deg(x_i) < r, ~\deg_F(x_i^2) < m \\
    &\hspace{14em}{x}_1 + {x}_2 \equiv {x}_3 + {x}_4 \Mod{F}\}. \nonumber 
\end{align}
Again, for initial intervals we will write $E_F^*(m)$ instead of $E_F^*(\cI_m)$ for simplicity.

\subsection{Exponential function in function fields}\label{sec:notation_characters} Here we give a typical description of additive characters in $\F_q[T]$, mostly taken from \cite[Sections 2 and 3]{Hayes1966}. This in turn allows us to describe all additive characters of $\F_q[T]/\langle F\rangle$. When $F$ is irreducible, this gives a very natural description of all additive characters in a finite field. 

We define the usual absolute value on $\F_q(T)$ as
\begin{align*}
    \left|\frac{g}{h}\right| =
    \begin{cases}
    q^{\deg g - \deg h}, & g \neq 0\\
    0, & g = 0.
    \end{cases}
\end{align*}
The completion of $\F_q(T)$ with respect to this absolute value is the field of Laurent series in $1/T$, 
$$\F_q((1/T)) = \left\{\sum_{i=-\infty}^na_iT^i~:~n \in \Z,~a_i \in \F_q,~a_n \neq 0\right\}.$$
Additionally, the absolute value extends to this space in the expected way, 
$$\bigg{|} \sum_{i=-\infty}^na_iT^i \bigg{|} = q^n. $$
We will set $\Ki := \F_q((1/T))$ for simplicity of notation. 

On $\Ki$ we have the non-trivial additive character 
\begin{align*}
    e\left(\sum_{i=-\infty}^na_iT^i\right) =\exp\bigg{(}\frac{2\pi i}{p}\text{Tr}(a_{-1}) \bigg{)}\nonumber
\end{align*}
where $p$ is the characteristic of $\F_q$ and $\text{Tr}: \F_q \to \F_p$ is the absolute trace. If $f  = \sum a_iT^i$ will write $[f]_{-1} = a_{-1}$.

If we look at the function
\begin{align*}
    e_{F}(x) = e({x}/{F}{)},
\end{align*}
for any $x,y \in \F_q[T]$ if $x \equiv y \Mod{F}$ then $e_F(x) = e_F(y)$. Thus we can consider $e_F$ as an additive character of $\F_q[T]/\langle F \rangle $. Slightly more generally, it is easy to see that
\begin{align*}
    e_{F, \lambda}(x) = e({x\lambda}/{F}{)}
\end{align*}
 describes all additive characters modulo $F$ as $\lambda$ runs through $\deg \lambda < \deg F$. 

\section{Main results}
We recall that we set $F \in \F_q[T]$ with $\deg F = r$. Throughout this section, $m$ and $n$ will denote positive integers $m,n\leq r$, and $\cI_m$ and $\cI_n$ intervals as in \eqref{eq:intervals}. 

When comparing different results, we will often need to state that one bound is stronger than another if, say, $x > y(1+\epsilon)$ for some sufficiently large variables $x$ and $y$ and some fixed $\epsilon > 0$. For shorthand we will denote this by $x >_{\epsilon} y$, to avoid writing $(1+\epsilon)$ every time. 

\subsection{Bilinear forms with Kloosterman sums}

For $s,t \in \F_{q}[T]$ we define the Kloosterman sum 
$${K}_F(s, t) = \sum_{\substack{\deg x < r \\ (x,F) = 1}}e_F(sx + t\ov{x}).$$
We emphasize that if $F$ is irreducible then this is the typical definition of a Kloosterman sum over a finite field. Weil's bound (and its natural extension to composite $F$ given in Lemma \ref{lem:weil_bound}) implies
\begin{align}\label{eq:weil_bound}
    |K_F(s,t)| \leq q^{r/2 + \deg(s,t,F)/2 + o(r)}.
\end{align}

Here we consider bilinear forms of these Kloosterman sums
$$\sum_{s \in \cS}\sum_{t \in \cT}\alpha_s\beta_tK_F(s,t)$$
for sequences of complex weights $\balpha$ and $\bbeta$ on finite sets $\cS, \cT \subseteq \F_q[T]$. Most often we will consider when $\cS$ and $\cT$ are intervals as in (\ref{eq:intervals}). In general, we wish to demonstrate cancellation in these bilinear forms beyond just that implied by (\ref{eq:weil_bound}), especially in the P{\'o}lya-Vinogradov range when $m,n\leq r/2$. 

If $\cT$ consists of exactly a full-set of residue classes modulo $F$, then basic orthogonality relations imply
\begin{align*}
   \sum_{s \in \cS}\sum_{t \in \cT}K_F(s,t) &=  \sum_{s \in \cS}\sum_{\deg t < r}K_F(s,t) \\
    &= \sum_{s \in \cS}\sum_{\substack{\deg x < r \\ (x,F) = 1} }e_F(sx)\sum_{\deg t < r}e_F(t\ov{x})
    = 0. 
\end{align*}
With this in mind, we firstly consider cancellations in shorter sums over intervals
$$\mathcal{K}_{F,a}(\cI_m,\cI_n) = \sum_{s \in \cI_m}\sum_{t \in \cI_n}K_F(s,at) $$
for $a \in \F_q[T]$. In particular, we wish to improve upon the trivial bound
\begin{align}\label{eq:kloos_trivial_1}
    \big{|}\mathcal{K}_{F,a}(\cI_m,\cI_n)\big{|} \leq q^{n+m+r/2 + o(r)}
\end{align}
implied by (\ref{eq:weil_bound}). 
Sums of this type were also considered in \cite{MacourtShparlinski2019}. Although their results only dealt with irreducible $F$ and required $q^r$ to be an odd power of a prime $p$, their sums were over arbitrary subspaces as opposed to intervals. Regardless, for initial intervals with $n \leq m$ and $a \not\equiv 0 \Mod{F}$, \cite[Theorem 1.1]{MacourtShparlinski2019} gives
\begin{align}\label{eq:Kloosterman_MacShpar1}
    |\mathcal{K}_{F,a}(\cI_m, \cI_n) | \ll q^{m+n}(q^{52r/153} + q^{831r/832-831n/832} + q^{r - 761n/760}).
\end{align}

We can improve upon this with the following. 
\begin{thm}\label{thm:BilinearKloosterman1}
For any integers $m,n \leq r$, intervals $\cI_m$ and $\cI_n$ as in (\ref{eq:intervals}) and any $a \in \F_q[T]$ we have
$$|\mathcal{K}_{F,a}\left(\cI_m,\cI_n\right)| ~\leq q^{n+m}H_{F,a}(r-m, r-n), $$
and Lemma \ref{lem:hyperbola} gives the bound 
$$H_{F,a}({r-m}, {r-n}) \leq q^{o(r)}\left(q^{r-m-n} +1\right).$$
\end{thm}

This improves upon the trivial bound (\ref{eq:kloos_trivial_1}) when $n+m >_\epsilon r/2 $, and in the case $m = n = r/2$ we have a saving of $q^{r/2 + o(r)}$ over the trivial bound. It also improves upon (\ref{eq:Kloosterman_MacShpar1}) for all ranges of $m$ and $n$.

We next consider \textit{Type-I} sums of the form
$$\mathcal{K}_{F,a}(\balpha; \cS, \cI_n) = \sum_{s \in \cS}\sum_{t\in \cI_n}\alpha_sK_F(s,at) $$
for a sequence of complex weights $\balpha$ on an arbitrary set $\cS \subseteq \{\deg x < r\}$. Again, we have a trivial bound 
\begin{align}\label{eq:kloos_trivial_2}
    \mathcal{K}_{F,a}(\balpha; \cS,\cI_n)
    &\leq \|\balpha\|_1q^{n+r/2 + o(r)}
\end{align}
implied by (\ref{eq:weil_bound}). 
Sums of this type were also considered in \cite{MacourtShparlinski2019}. Their result
\cite[Theorem 1.2]{MacourtShparlinski2019} essentially implies 
\begin{align}\label{eq:Kloosterman_MacShpar2}
    |\mathcal{K}_{F,a}(\balpha; \cS,\cI_n) | \ll \|&\balpha\|_1^{1/2}\|\balpha\|_2^{1/2} \nonumber \\(&q^{n+103r/204} + q^{313n/1248 + 1247r/1248} + q^{71n/285 + r}).
\end{align}

We can again improve upon this with the following. 
\begin{thm}\label{thm:BilinearKloosterman2}
Let $\balpha$ denote a sequence of complex weights on an arbitrary set $\cS \subseteq \{\deg x < r\}$. For any positive integer $n \leq r$ and interval $\cI_n$ as in (\ref{eq:intervals}) we have that
 \begin{align}\label{eq:BilinearKloosSum1_general}
        |\mathcal{K}_{F,a}(\balpha;\cS,\cI_n) |  &\leq \|\balpha\|_2q^{n+r/2}H_{F,a}(r, r-n)^{1/2},
\end{align}
 and Lemma \ref{lem:hyperbola} gives the bound 
 $$H_{F,a}(r, r-n)^{1/2} \leq q^{r/2-n/2+o(r)}.$$
\end{thm}
\begin{rem}\label{rem:kloos_2}
The proof of Theorem \ref{thm:BilinearGaussSum1} below can be carried out essentially identically but with $\cK_{F,a}(\balpha; \cS, \cI_n)$ to produce
\begin{align}\label{eq:BilinearKloosSum1_irreducible}
    |\cK_{F,a}(\balpha; \cS, \cI_n)|
    &\leq \|\balpha\|_1^{1/2}\|\balpha\|_2^{1/2}q^{n + r/4}E_{F}^\inv(r-n)^{1/4}
\end{align}
with bounds for $E_{F}^\inv(r-n)^{1/4}$ given in Lemma \ref{lem:inverse_energy_both} as either of 
\begin{align*}
    &q^{o(r)}\left(q^{3r/4 - 7n/8} + q^{r/2-n/2 }\right),\\
    &q^{o(r)}\left(q^{3r/4-n}+q^{5r/8-n/2} \right).
\end{align*}
But, this bound would only hold for irreducible $F$ and $\gcd(a,F) = 1$, as opposed to arbitrary $F$ and $a$ as in (\ref{eq:BilinearKloosSum1_general}). 
\end{rem}
Firstly, (\ref{eq:BilinearKloosSum1_irreducible}) always improves upon the previous results (\ref{eq:Kloosterman_MacShpar2}). 
If $\balpha$ is supported on a set of size $q^m$ and $\|\balpha\|_\infty \leq q^{o(r)}$, then (\ref{eq:BilinearKloosSum1_general}) is non-trivial when $m + n >_\epsilon r$, while (\ref{eq:BilinearKloosSum1_irreducible}) is non-trivial if
$$2r/3 \leq  n < r \text{ and } 2m + 7n >_\epsilon 4r$$
or 
$$r/3 \leq  n \leq 2r/3 \text{ and } m + 2n >_\epsilon r$$
or 
$$r/4 <  n \leq r/3 \text{ and } 2m + 4n >_\epsilon 3r.$$
In particular, this presents a savings of $q^{r/16 + o(r)}$ over the trivial bound when $m=n=r/2$. But we can do better than this for arbitrary $F$, if we allow for slightly less flexibility on the shape of $\cS$.  We thus finally consider \textit{Type-I} sums over intervals of the form 
$$\mathcal{K}_{F,a}(\balpha; \cI_m,\cI_n) = \sum_{s \in \cI_m}\sum_{t\in \cI_n}\alpha_sK_F(s,at). $$
The trivial bound (\ref{eq:kloos_trivial_2}) again holds here, and we can improve upon this in certain ranges with the following. 
\begin{thm}\label{thm:BilinearKloosterman3}
Let $D = \gcd(a,F)$ and $d = \deg D$. For simplicity, let $F_0 = F/D$ and $a_0 = a/D$. For any integers $n,m \leq r$, intervals $\cI_m$ and $\cI_n$ as in (\ref{eq:intervals}), and a sequence of complex weights $\balpha$ on $\cI_m$ we have that  
 \begin{align*}
     |\mathcal{K}_{F,a}(\balpha;& ~\cI_m,\cI_n)  |
 \\
     &\leq  \|\balpha\|_2q^{n+m/2+d/2}A_{F_0, \ov{a_0}}(r-d-n, r-m)^{1/2}
\end{align*}
and equations (\ref{eq:average_inverse_1}) and (\ref{eq:average_inverse_3}) give that $A_{F_0, \ov{a_0}}(r-d-n, r-m)^{1/2}$ is bounded above by either of 
\begin{subequations}
\begin{align}
&q^{o(r)}\left(q^{r/2-d/2-n/2} + q^{r/2-m/2}+ q^{r-d/2-m/2-3n/4}\right), \label{eq:BK3_1}\\
    &q^{o(r)}\left(q^{r/2-d/2-n/2}+ q^{r-d/2-m/2-n} + q^{r-3d/4-m/4-n}\right)\label{eq:BK3_3}.
\end{align}
\end{subequations}
Additionally if $q$ is odd then by \eqref{eq:average_inverse_2} we also have the bound 
\begin{subequations}
\begin{align}
    &q^{o(r)}\left(q^{r/2 - d/2 - n/2} + q^{r-d/2  - m/2-n} + q^{3r/4-d/4-m/2}\right).\label{eq:BK3_2}
\end{align}
\end{subequations}
\end{thm}
For comparisons to other bounds, we will again suppose that $\|\balpha\|_\infty \leq q^{o(r)}$. Firstly, when $m=n=r/2$ and $d = 0$, (\ref{eq:BK3_1}) and (\ref{eq:BK3_3}) are equivalent and give the best bound of $q^{11r/8 + o(r)}$, which yields a savings of $q^{r/8 + o(r)}$ over the trivial bound (a larger savings than (\ref{eq:BilinearKloosSum1_irreducible})). 

A full comparison of (\ref{eq:BK3_1}), (\ref{eq:BK3_3}) and (\ref{eq:BK3_2}) for all ranges of $m,n$ and $d$ would be quite tedious, but each presents a non-trivial bound in certain situations. As mentioned, (\ref{eq:BK3_1}) and (\ref{eq:BK3_3}) perform best in the P{\'o}lya-Vinogradov range. If we fix $d=0$ and $n=r/2$ then (\ref{eq:BK3_1}) gives the most flexibility on $m$, being the best bound and non-trivial when $r/4 < m < 3r/4$. But (\ref{eq:BK3_3}) does better for $d > 0$, presenting a savings of up to $q^{r/4 + o(r)}$ over the trivial bound when $m=n=d = r/2$. An example of when (\ref{eq:BK3_2}) provides the best bound, and is also non-trivial, is when $d = 0$, $m \geq r/2$, $n \leq r/3$ and $m+2n >_\epsilon r$.

\begin{rem}
The proofs in Section \ref{sec:proofs} indicate that the bounds in Theorems \ref{thm:BilinearGaussSum1}, \ref{thm:BilinearGaussSum2} and \ref{thm:BilinearGaussSum3} still hold for more general bilinear forms of weighted Kloosterman sums
$$\sum_{\substack{\deg x < r \\ (x,F) = 1}}\gamma_xe_F(sx + t\ov{x}),$$
for any sequence of complex weights $|\gamma_x| \leq 1$. 
\end{rem}

\subsection{Bilinear forms with Gauss sums}
Motivated by the results of \cite{Shparlinski2022_Gauss}, we also consider very similar questions for bilinear forms of Gauss sums. For $s,t \in \F_{q}[T]$ we define the Gauss sum 
$$G_F(s,t) = \sum_{\deg x < r}e_F(sx + tx^2)$$
and consider bilinear forms 
$$ \sum_{s \in \cS}\sum_{\substack{t \in \cT \\ t \not\equiv 0(F)}}\alpha_s\beta_t G_F(s, at) $$
for $F$ irreducible and $F \nmid a$, and for sequences of complex weights $\balpha$ and $\bbeta$ on finite sets $\cS, \cT \subseteq \F_q[T]$. 
We firstly consider \textit{Type-I} sums
$$\cG_{F,a}(\balpha; \cS, \cI_n) = \sum_{s\in \cS}\sum_{\substack{t \in {I}_n \\ t \not\equiv 0 (F)} } \alpha_s G_F(s, at)$$
and are interested in improving upon the trivial bound
\begin{align}\label{eq:Gauss_trivial_1}
    |\cG_{F,a}(\balpha; \cS, \cI_n)|
    \leq \|\balpha \|_1q^{n+ r/2}
\end{align}
implied Lemma \ref{lem:app:Gauss_1}.

We firstly have the following. 
\begin{thm}\label{thm:BilinearGaussSum1}
Let $F$ be irreducible and $q$ be odd. Let $\balpha$ denote a sequence of complex weights on an arbitrary set $\cS \subseteq \{\deg x < r\}$.  For any positive integer $n \leq r$, interval  $\cI_n$ as in (\ref{eq:intervals}) and any $a \in \F_{q}[T]$ coprime with $F$ we have 
\begin{align*}
     |\cG_{F,a}(\balpha; \cS, \cI_n)| \leq  \|\balpha\|_1^{1/2}\|\balpha\|_2^{1/2}q^{n + r/4}E_{F}^\sqrt(r-n)^{1/4}
\end{align*}
and Lemma \ref{lem:sqrt_energy} gives the bound
$$E_F^{\sqrt}(r-n)^{1/4} \leq q^{o(r)}\left(q^{3r/4 - 7n/8} + q^{r/2-n/2 }\right).$$
\end{thm}
This bound is very similar to that given in Remark \ref{rem:kloos_2}, and if $\|\balpha\|_\infty \leq q^{o(r)}$ and $\balpha$ is supported on a set of size $q^m$ it is non-trivial when
$2r/3 \leq  n < r \text{ and } 2m + 7n >_\epsilon 4r$,
or  when
$r/3 \leq  n \leq 2r/3 \text{ and } m + 2n >_\epsilon r$. Again it presents a savings of $q^{r/16 + o(r)}$ over the trivial bound (\ref{eq:Gauss_trivial_1}) when $m=n=r/2$. 

We next consider \textit{Type-II} sums of the form 
$$\cG_{F,a}(\balpha, \bbeta;\cS,\cI_n) = \sum_{s \in \cS}\sum_{\substack{t \in \cI_n \\ t \not\equiv 0 (F)}} \alpha_s\beta_t G_F(s, at). $$
 In this case, Lemma \ref{lem:app:Gauss_1} yields the trivial bound 
\begin{align*}
|\cG_{F,a}(\balpha, \bbeta;\cS,\cI_n)|
    &\leq \|\balpha \|_1\|\bbeta\|_1q^{r/2}.
\end{align*}

We can show the following. 
\begin{thm}\label{thm:BilinearGaussSum2}
Let $F$ be irreducible and $q$ be odd. Let $n\leq r$ be a positive integer and $\cI_n$ as in (\ref{eq:intervals}). Let $\cS \subseteq \{\deg x < r\}$ denote an arbitrary set and $ \balpha$ and $\bbeta$ sequences of complex weights on $\cS$ and $\cI_n$ respectively. For any $a \in \F_{q}[T]$ such that $\gcd(a,F) = 1$,
    \begin{align*}
|\cG_{F,a}(\balpha, \bbeta;\cS,\cI_n)| \leq \|\balpha\|_1^{1/2}\|\balpha\|_2^{1/2}\|\bbeta\|_\infty q^{3r/4}E_F^{\inv}(\cI_n)^{1/4}, 
\end{align*}
and Lemma \ref{lem:inverse_energy_both} states that $E_F^{\inv}(\cI_n)^{1/4}$ is bounded above by either of
\begin{subequations}
   \begin{align}
    &q^{o(r)}\left(q^{n-r/4}+ q^{n/2 + r/8}\right)\label{eq:BG2_1},\\
&q^{o(r)}\left(q^{7n/8-r/8}+ q^{n/2}\right)\label{eq:BG2_2} .
\end{align} 
\end{subequations}

\end{thm}
Again for comparison, suppose $\balpha$ is supported on a set of size $q^m$ and $\|\balpha\|_\infty, \|\bbeta\|_\infty  \leq q^{o(r)}$. When $n=m=r/2$, (\ref{eq:BG2_2}) presents a savings of $q^{r/16 + o(r)}$ over the trivial bound.
In general when $m=n$, these are non-trivial when $m=n >_\epsilon r/3$: (\ref{eq:BG2_2}) is best  $r/3 <_\epsilon n \leq 2r/3$, while (\ref{eq:BG2_1}) is best for $ n > 2r/3$.  

%%$$3r/4 \leq n < r$$ 
%%$$2r/3 \leq n \leq 3r/4 \text{ and } 2m + 4n >_\epsilon 3r $$
%%$$r/3 \leq n \leq 2r/3 \text{ and } 2m + n >_\epsilon r$$ 
%%$$n \leq r/3 \text{ and } m + 2n >_\epsilon r.$$

Finally, we consider \textit{Type-II} sums of the form 
$$\cG_{F,a}(\balpha, \bbeta;\cI_m,\cI_n) = \sum_{s \in \cI_m}\sum_{\substack{t \in \cI_n \\ t \not\equiv 0 (F)}} \alpha_s\beta_t G_F(s, at). $$

\begin{thm}\label{thm:BilinearGaussSum3}
Let $F$ be irreducible and $q$ be odd. Let $m,n\leq r$ be positive integers, $\cI_m$ and $\cI_n$ as in (\ref{eq:intervals}), and $\balpha$ and $\bbeta$ sequences of complex weights on $\cI_m$ and $\cI_n$, respectively. Let $a \in \F_{q}[T]$ such that $\gcd(a,F) = 1$. Then
\begin{align*}
\cG_{F,a}(\balpha, \bbeta;\cI_m,\cI_n) =& \|\balpha\|_2\|\bbeta\|_\infty q^{5r/8 + n/2}E_F^{\inv}(\cI_n)^{1/8}E_F^{\sq}(\cI_m)^{1/8}.
\end{align*}
Lemma \ref{lem:inverse_energy_both} states that 
\begin{subequations}
    \begin{align*}
    %%&q^{o(r)}\left(q^{n/2-r/8}+ q^{n/4 + r/16}\right),\label{eq:BG3_1}\\ doesnt work...
    E_F^{\inv}(\cI_n)^{1/8} \leq q^{o(r)}\left(q^{7n/16-r/16}+ q^{n/4 }\right)
\end{align*}
\end{subequations}
and Lemma \ref{lem:squares_energy_general_improved} gives the bound
\begin{align*}
    E_F^{\sq}(\cI_m)^{1/8}
    &\leq 
    q^{o(r)}\left(
    q^{m/2-r/8} + q^{m/4}\right)
\end{align*}
\end{thm}
Firstly, although Lemma \ref{lem:inverse_energy_both} states two bounds, we have only applied one as the other does not seem to give any improvement over Theorem \ref{thm:BilinearGaussSum2}. 

Again, it would be tedious to compare all possible ranges of $m$ and $n$, but we highlight a few main points. Suppose $\|\balpha\|_\infty, \|\bbeta\|_\infty \leq q^{o(r)}$. If $m=n=r/2$, this bound presents an additional savings of $q^{3r/32 + o(r)}$ over the trivial bound. In general if $m=n$ then this is non-trivial if $m=n >_\epsilon r/4$ (as opposed to only $m=n >_\epsilon r/3$ in Theorem \ref{thm:BilinearGaussSum2}). But, Theorem \ref{thm:BilinearGaussSum2} is better for $m=n \geq r/2$.

\section{Preliminaries}
\subsection{Exponential sums over intervals} 
As noted previously, initial intervals in $\F_q[T]$ are closed under addition. This additive structure allows for simple evaluation of exponential sums, such as the following orthogonality relation given in \cite[Theorem 3.7]{Hayes1966}. 
\begin{lem}\label{lem:Hayes}
Let $u \in \Ki$ with $|u|<1$. Then for any positive integer $m$ we have 
\begin{align*}
    \sum_{\deg x < m}e(xu) 
    = 
    \begin{cases}
    q^m, &|u| < q^{-m}\\
    0, &\text{otherwise}.
    \end{cases}
\end{align*}
\end{lem}
From this we can derive a simple corollary, which is crucial for our results. Although upon first inspection this might appear different to results such as \cite[equation (8.6)]{IwaniecKowalski2004} for exponential sums over intervals in $\Z$, in applications these yield very similar results. 
\begin{cor}\label{cor:charsum_subspace}
Let $u \in \F_q[T]$. Then for any positive integer $m$ we have 
\begin{align*}
    \sum_{\deg x < m}e_F(xu) 
    = 
    \begin{cases}
    q^m, &\deg_Fu < r-m\\
    0, &\text{otherwise}.
    \end{cases}
\end{align*}
\end{cor}
\begin{proof}
We can write $u = tF +u'$ for unique $t,u' \in \F_q[T]$ with $\deg u' < r$. Thus we have 
\begin{align*}
    \sum_{\deg x < m}e_F(xu) 
    &= \sum_{\deg x < m}e\big{(}xu'/F \big{)}\\
    &= 
    \begin{cases}
    q^m, &|u'/F| < q^{-m}\\
    0, &\text{otherwise}
    \end{cases}
\end{align*}
where the last line is a direct application of Lemma \ref{lem:Hayes}. The result then follows immediately. 
\end{proof}

Another (perhaps more general) way to see that Corollary \ref{cor:charsum_subspace} holds is to note that in $\F_q[T]/\langle F\rangle$, the $\F_q$-subspace $\{\deg x < m\}$ is dual to $\{\deg x < r-m\}$ with respect to the bilinear form 
$$\langle x,y\rangle  = \textup{Tr}\left[\frac{xy}{F} \right]_{-1}$$
where $[xy/F]_{-1}$ is the coefficient on $1/T$ in the series expansion of $xy/F \in \Ki$ (see Section \ref{sec:notation_characters}).

\subsection{Diophantine approximation}
Next, we will need a simple result regarding solutions to the congruence $\lambda x_1 \equiv x_2 \Mod{F}$, and it is most simply seen as a corollary to the following from \cite[Theorem 4.3]{Hayes1966}. This can be considered an $\F_q[T]$ analogue Dirichlet's approximation theorem. 
\begin{lem}\label{lem:dirichlet_approx}
For any $\xi \in \Ki$ satisfying $|\xi| < 1$, and any positive integer $k$, there exists some $g,h \in \F_q[T]$ with $\gcd(g,h) = 1$ and $\deg g < \deg h \leq k $ such that 
$$\bigg{|}\xi - \frac{g}{h} \bigg{|} \leq \frac{1}{q^{\deg h + k+1}}.$$
\end{lem}

We now have the following as a simple corollary. 
\begin{cor}\label{lem:cor_dirichlet}
Let $\lambda \in \F_{q}[T]$. Then for any positive integer $k \leq  r$ there exists $x_1,x_2 \in \F_q[T]$ with $x_1 \neq 0$ satisfying $\deg x_1 \leq k$ and $\deg x_2 \leq r-k-1$ such that $\lambda x_1 \equiv x_2 \Mod{F}$. 
\end{cor}
\begin{proof}
We may suppose $\deg \lambda < r$. Lemma \ref{lem:dirichlet_approx} implies there exists some coprime $x_1,g\in \F_q[T]$ satisfying $\deg g < \deg x_1 \leq k$ such that
$$\bigg{|}\frac{\lambda}{F} - \frac{g}{x_1} \bigg{|} \leq {q^{-\deg x_1 - k-1}}. $$
The fact that $\deg g < \deg x_1$ of course means $x_1 \neq 0$. Rearranging yields
$$|\lambda x_1 - gF| \leq q^{r-k-1}. $$
This implies that $\lambda x_1$ has a representative modulo $F$ of degree bounded above by $r-k-1$, which we may call $x_2$. 
\end{proof}

\subsection{Divisor bound for polynomials}

Finally, we will make repeated use of the following found in \cite[Lemma 1]{CS2013}. 
\begin{lem}\label{lem:divisor_bound}
The number of divisors of any $x \in \F_q[T]$ is $q^{o(\deg x)}$. 
\end{lem}
Lemma \ref{lem:divisor_bound} also implies
\begin{align}\label{eq:sum_gcd}
\sum_{\substack{\deg y < m \\ y \neq 0}}q^{\deg(x,y)}
&\leq \sum_{\substack{d|x \\ \deg d < m}}q^{\deg d}\sum_{\substack{\deg y < m-\deg d }}1 
\leq q^{m + o(\deg x)},
\end{align}
which will be used throughout.

\section{Counting solutions to modular congruences}\label{sec:counting_solution}
This section will focus on bounding the counting functions introduced in Section \ref{sec:notation_counting}, which may be of independent interest. Throughout, let $m$ and $n$ denote positive integers, and again  $\cI_m$ and $\cI_n$ as in (\ref{eq:intervals}). 
\subsection{Hyperbolas}
We first consider bounding the number of points on the modular hyperbola $H_{F,a}(\cI_m, \cI_n)$ as in (\ref{eq:def:hyperbola}). The following bound is found in \cite[Theorem 4]{CS2013}, which applies to irreducible $F$. The fact that the stated bound holds also for $H_{F,a}(\cI_m, m)$ is contained within the proof of \cite[Theorem 4]{CS2013}, and we sketch the details here. 

\begin{lem}\label{lem:hyperbola_shparlinski_cill}
Let $F$ be irreducible and suppose $\gcd(a,F) = 1$. For any positive integer $m \leq r$ and interval $\cI_m$ as in (\ref{eq:intervals}), 
$$H_{F,a}(\cI_m, \cI_m) \leq q^{o(m)}(1 + q^{3m/2-r/2}).$$
Additionally, we have that the same bound holds for $H_{F,a}(\cI_m, m)$. 
\end{lem}
\begin{proof}
    Two different bounds are given in \cite[Theorem 4]{CS2013}. The difference in these bounds depends on different bounds for $\deg t$ obtained in their proof (above their equation (8)). They also explain this at the end of their proof. 

    In the case of bounding $H_{F,a}(\cI_m, m)$, carrying out a change of variables as they do at the start of their proof, their equation (7) becomes, in our notation, 
    $$x_1x_2 + bx_2 \equiv c \Mod{F},$$
    for some $b,c \in \F_q[T]$. Now applying Corollary \ref{lem:cor_dirichlet}, for any integer $s \leq r$ we can find some $t, v_0$ with $\deg t < s+1 $ and $\deg v_0 < r-s$ such that $bt \equiv v_0 \Mod{F}$. In their notation, we can also take $u_0 = 0$. With this bound on $\deg t$, carrying out the rest of their proof, in the same manner, gives the desired bound. 
\end{proof}

This can be improved upon this for initial intervals, for arbitrary $F$. 
\begin{lem}\label{lem:hyperbola}
For any positive integers $n$ and $m$, 
$$H_{F,a}(m,n) \leq q^{o(m+n)}(q^{m+n-r} + 1). $$
\end{lem}
\begin{proof}
We note that we can assume that $\deg a < r$. The congruence implies $x_1x_2 = a + tF$ for some $t \in \F_q[T]$. Since $\deg(x_1x_2) < m+n$, we have at most $O(q^{n+m-r} + 1)$ choices for $t$. For each choice of $t$, we have that $x_1$ and $x_2$ must take on one of the $q^{o(m+n)}$ divisors of $a + tF$ (by Lemma \ref{lem:divisor_bound}).
\end{proof}

\subsection{Sums of inverses}
Next we consider bounding $I_{F,a}(\cI_m)$ as in (\ref{eq:def:sum_inverse}). By adapting some ideas from \cite[page 367]{HeathBrown1978} we firstly have the following.  
\begin{lem}\label{lem:inverse_via_heathbrown}
Let $d = \deg \gcd(a, F)$. For any positive integer $m \leq r$,
\begin{align*}
    I_{F,a}(m) \leq q^{o(r)}(1 + q^{3m/2 - r/2} + q^{2m + d - r}).
\end{align*}
Furthermore, if $F$ is irreducible and $d=1$ then this bound applies to $I_{F,a}(\cI_m)$.
\end{lem}
\begin{proof}
We can suppose that $d < r$, since otherwise the result holds trivially. 
Let $k = \min\{r-d - 1, \lceil (r+m)/2 \rceil\}$. By Corollary \ref{lem:cor_dirichlet} there exists $u,v \in \F_q[T]$ with $u \neq 0$ such that
$$a u \equiv v \Mod{F},~~ \deg u \leq k, ~\deg v \leq r-k-1.$$ 
Also note that we can say $v \neq 0$ since $u \neq 0$ and $\deg u < r-d$. 
%To see this, for the sake of contradiction suppose $v = 0$. We know $u \neq 0$ and since $d < r$ we know $\lambda \neq 0$. We can write $F = h\gcd(d,F)$ for some $h$ which is coprime to $\lambda$ and by construction $\deg h = r-d$. If $v=0$ then $F | u\lambda$. So this implies $u\lambda = h\gcd(\lambda, F)t$ for some $t \in \F_q[T]$. But since $\gcd(\lambda, h) = 1$ we can say $h |u$, but $\deg u < r-d = \deg h$ - a contradiction. 
Now $\ov{x_1} + \ov{x_1} \equiv a \Mod{F}$ implies $u(x_1 + x_2) \equiv vx_1x_2 \Mod{F}$, or that 
$$ u(x_1 + x_2) - vx_1x_2  = tF$$
for some $t \in \F_q[T]$. The bounds on $x_1, x_2, u$ and $v$ imply that there are at most 
$1 + q^{3m/2 - r/2} + q^{2m + d - r} $ choices for $t$ (up to a constant). For each choice of $t$ we have 
$$(vx_1-u)(vx_2-u) = u^2 + tvF $$
and thus each of $x_1, x_2$ must correspond to one of the $q^{o(r)}$ divisors of $u^2 +tvF$, giving the desired result. 

The fact that this bound holds for any interval in the case of irreducible $F$ is contained at the start of the proof of \cite[Theorem 2.2]{BS2022}; the initial congruence implies 
$$(x_1-\ov{a})(x_2-\ov{a}) \equiv \ov{a}^2 \Mod{F}$$
after which one may apply Lemma \ref{lem:hyperbola_shparlinski_cill}. 
\end{proof}
Next, by adapting a few ideas from \cite{KerrShparlinskiWuXi2022}, we can use some of the results from Appendix \ref{appendix:kloos_and_gauss} to give an improvement upon Lemma \ref{lem:inverse_via_heathbrown} when $m > 2r/3$. 
\begin{lem}\label{lem:inverse_via_explicit_kloos}
Assume $q$ is odd. Let $d = \deg \gcd(a,F)$. For any positive integer $m \leq r$ and interval $\cI_m$ as in (\ref{eq:intervals}) we have 
$$I_{F,a}(\cI_m) \leq q^{o(r)}\left(q^{2m - r} + q^{m + d/2 - r/2} + q^{r/2} \right).$$
\end{lem}
\begin{proof}
We will firstly point out that if $d = 1$ the result follows much more readily from the Weil-Estermann bound (Lemma \ref{lem:weil_bound}) and orthogonality, using an argument similar to that found in \cite[Theorem 13]{S2012}. But we can deal with larger $d$ with a bit more work. By multiple applications of Corollary~\ref{cor:charsum_subspace}, 
    \begin{align}\label{eq:inverse_as_char_sum}
        I_{F,a}&(\cI_m) \nonumber \\
        &= q^{2n-3r}\sum_{\substack{(x_1, x_2) \in \F_q[T]^2 \\ \deg x_i < r \\ (x_i, F) = 1}}\left(\sum_{\deg u < r-m}e_F(u(x_1-s_0)) \right)\nonumber \\
        &\quad\quad\times\left(\sum_{\deg v < r-m}e_F(v(x_2-s_0)) \right) \sum_{\deg t < r}e_F(t(\ov{x_1} + \ov{x_2} - a)).
    \end{align}
To see this, for a given $(x_1, x_2)$ in the outer sum, the inner sum over $t$ is $0$ unless $\ov{x_1} + \ov{x_2} \equiv a \Mod{F}$, in which case it is equal to $q^r$. The two sums in brackets are also $0$, unless $\deg_F(x_1 - s_0) < m$ and $\deg_F(x_2 - s_0)< m$, in which case their product is equal to $q^{2r-2m}$. Thus diving by $q^{3r-2m}$ and summing over all $(x_1, x_2)$ gives the desired count. 

Rearranging (\ref{eq:inverse_as_char_sum}) now yields 
    \begin{align*}
        &I_{F,a}(\cI_m) \nonumber \\
        &= q^{2m-3r}\sum_{\substack{\deg u < r-m \\ \deg v < r-m}}e_F(-s_0(u+v))\sum_{\deg t < r}K_F(u,t)K_F(v,t)e_F(-at).
    \end{align*}
We can now directly apply  Lemma \ref{lem:T_bound_final} to see 

$$I_{F,a}(\cI_m) \leq q^{2m-3r/2 + o(r)}\sum_{\substack{\deg u < r-m \\ \deg v < r-m}}q^{\deg(u-v,a,F/(u,v,F))/2 + \deg(u,v,F)/2}.$$
By splitting into cases and simplifying we have 
\begin{align}\label{eq:I_in_terms_of_S1_S4}
    I_{F,a}(\cI_m)
    &\leq q^{2m-3r/2 + o(r)}\left(S_1 + S_2 + S_3\right)
\end{align}
where 
\begin{align*}
S_1 &= \sum_{\substack{\deg u < r-m \\ u \neq 0}}\sum_{\substack{\deg v < r-m \\ v \neq 0, u}}q^{\deg(u-v, F)/2 + \deg(u,v)/2},\\
S_2 &= \sum_{\substack{\deg u < r-m \\ u \neq 0}}q^{\deg(a, F)/2 + \deg(u,F)/2},\\
S_3 &= \sum_{u=v=0}q^{r/2} = q^{r/2}.
\end{align*}
To deal with $S_1$, we split the exponents up and write 
\begin{align*}
    S_1
    &\ll \sum_{\substack{\deg u < r-m \\ u \neq 0}}\sum_{\substack{\deg v < r-m \\ v \neq 0, u}}\left(q^{\deg(u-v, F)} + q^{\deg(u,v)}\right).
\end{align*}
For a fixed $u$, as $v$ runs over $\deg v < r-m$ we see that $u-v$ runs over this same set. Thus
\begin{align*}
    S_1
    &\ll \sum_{\substack{\deg u < r-m \\ u \neq 0}}\sum_{\substack{\deg v < r-m \\ v \neq 0}}\left(q^{\deg(v, F)} + q^{\deg(u,v)}\right).
\end{align*}
Now applying Lemma \ref{eq:sum_gcd} implies 
$$S_1 \leq q^{2r-2m + o(r)}.$$
Similarly applying Lemma \ref{eq:sum_gcd} to $S_2$ gives 
\begin{align*}
S_2
&\leq 
q^{\deg(a, F)/2}\sum_{\substack{\deg u < r-m \\ u \neq 0}}q^{\deg(u,F)}\\
&\leq q^{d/2 + r-m + o(r)}.
\end{align*}
Combining the estimates for $S_1, S_2$ and $S_3$ in (\ref{eq:I_in_terms_of_S1_S4}) gives the desired result. 
\end{proof}

One could also consider averages over $I_{F,a}({\cI_m})$: namely $A_{F,a}(\cI_m, k)$ as in (\ref{eq:def:sum_inverse_average}). 
Direct applications of Lemmas \ref{lem:inverse_via_heathbrown} and \ref{lem:inverse_via_explicit_kloos}, in conjuction with (\ref{eq:sum_gcd}) yields 
\begin{align}\label{eq:average_inverse_1}
    A_{F,a}(m, k)
    \leq q^{o(r)}\left(q^{m} + q^k + q^{2m-r+k} + q^{3m/2-r/2+k} \right)
\end{align}
and 
\begin{align}\label{eq:average_inverse_2}
    A_{F,a}(\cI_m, k)
    \leq q^{o(r)}\left(q^m + q^{2m-r+k} + q^{r/2 + k} + q^{m + k - r/2} \right)
\end{align}
for $a$ coprime to $F$ (of course recalling (\ref{eq:average_inverse_2}) requires $q$ to be odd), but a more specialized argument can do better in certain ranges. 
\begin{lem}
For positive integers $m,k \leq r$, and $a \in \F_q[T]$ satisfying $\gcd(a,F) = 1$, we have 
\begin{align}\label{eq:average_inverse_3}
    A_{F,a}(m, k) &\leq q^{o(r)}\left(q^{2m - r/2 + k/2 } + q^{2m -r + k} + q^m\right).
\end{align}
\end{lem}
\begin{proof}
If $k > r$, then trivially we have the bound $q^{2m - r + k}$. Thus we now assume $k \leq r$, and of course we may also assume $\deg a < r$. Recall that we are counting the number of solutions to 
$$\ov{x_1} - \ov{x_2} \equiv ah \Mod{F}, ~~\deg h < k, ~~\deg x_1, \deg x_2 < m. $$
Note that for any $\deg h < k$, there are $q^k$ solutions to $h = h_1-h_2$ for $\deg h_1, \deg h_2 < k$. Thus 
$$A_{F,a}(m, k)  = Bq^{-k} $$
where $B$ denotes the number of solutions to 
$$\ov{x_1} - \ov{x_2} \equiv a(h_1-h_2) \Mod{F}, ~\deg h_1, \deg h_2 < k, ~\deg x_1, \deg x_2 < m. $$
If, for any $\deg t < r$, we define $B(t)$ to count the number of solutions to 
\begin{align}\label{eq:H(t)}
    \ov{x} + ah \equiv t \Mod{F}, ~~\deg h < k, ~\deg x < m 
\end{align}
then we have 
\begin{align}\label{eq:H_factor_max}
    B
    = \sum_{\deg t < r}B(t)^2
    \leq \max_{\deg t < r}B(t)\sum_{\deg t < r}B(t)
    \leq q^{k+m}\max_{\deg t < r}B(t).
\end{align}
So it suffices to bound $B(t)$ uniformly in $t$. Dividing (\ref{eq:H(t)}) by $a$ we have that $B(t)$ counts 
\begin{align}\label{eq:H(t)_2}
   \ov{a}\ov{x} + h \equiv \ov{a}t \Mod{F}, ~~\deg h < k, ~\deg x < m.  
\end{align}
Now Corollary \ref{lem:cor_dirichlet} implies there exists $v,u \in \F_q[T]$, with $u \neq 0$, satisfying
$$\ov{a}tu \equiv v \Mod{F}, ~\deg v < \lceil r/2 \rceil + \lfloor k/2 \rfloor, ~\deg u < \lceil r/2 \rceil - \lfloor k/2 \rfloor + 1.  $$
We fix such $u$ and $v$. Now by (\ref{eq:H(t)_2}) this implies 
$$x(v-hu) \equiv \ov{a}u \Mod{F}. $$
Also, set $b \equiv \ov{a}u \Mod{F}$ with $\deg b < r$. Moving to $\F_q[T]$, there exists some $w$ such that 
$$x(v-hu) = b + wF.$$
Using the bounds on $x,v,h$ and $u$ there are at most
$$O\left(q^{m + k/2 - r/2 } + 1 \right)$$
choices for $w$. But for each choice of $w$, we can say $x$ and $(v-hu)$ must be divisors of $\ov{a}u + w F$. Recall that $u,v$ are fixed. Applying \cite[Lemma 1]{CS2013}, this implies there are at most $q^{o(r)}$ choices for each of $x$ and $h$. Thus 
$$B(t) \leq q^{m - r/2 + k/2 + o(r)} + q^{o(r)}. $$
Since this bound is uniform in $t$, substituting into (\ref{eq:H_factor_max}) yields 
$$B \leq q^{2m - r/2 + 3k/2 + o(r)} + q^{k+m+o(r)} $$
as desired. 

\end{proof}

Finally, we consider the problem of bounding $E_{F}^{\inv}(\cI_m)$ as in (\ref{eq:def:energy_inv}). 
In \cite{BS2022} the observation is made that 
\begin{align*}
    E_{F}^{\inv}(\cI_m) 
    &= \sum_{\deg a < r}I_{F,a}(\cI_m)^2 \\
    &\leq q^{2m} + \max_{\substack{\deg b < r \\ b \neq 0}}I_{F,b}(\cI_m)\sum_{\deg a < r}I_{F,a}(\cI_m)\\
    &= q^{2m}( 1 + \max_{\substack{\deg b < r \\ b \neq 0}}I_{F,b}(\cI_m)).
\end{align*}
This immediately leads to non-trivial bounds for irreducible $F$ by applying Lemmas \ref{lem:inverse_via_heathbrown} and \ref{lem:inverse_via_explicit_kloos}, although we note that the second bound given here is also given in \cite[Theorem 2.2]{BS2022}

\begin{lem}\label{lem:inverse_energy_both}
Let $F$ be irreducible. For any positive integer $m \leq r$ and $\cI_m$ as in (\ref{eq:intervals}), $E_F^{\inv}(\cI_m)$ is bounded above by either of 
\begin{subequations}
\begin{align}
    & q^{o(r)}\left(q^{4m-r}+ q^{2m + r/2}\right), \label{eq:inverse_energy_1}\\
&q^{o(r)}\left(q^{7m/2-r/2}+ q^{2m}\right)\label{eq:inverse_energy_2}.
\end{align} 
\end{subequations}
We recall that (\ref{eq:inverse_energy_1}) requires $q$ to be odd, while (\ref{eq:inverse_energy_2}) holds for arbitrary $q$. 
\end{lem}

\subsection{Sums of squares and square roots}

Next, for $E_{F}^{\sq}(\cI_m)$ as in (\ref{eq:def:energy_sq}) we have the following which is stronger than a direct function field analogue of \cite[(7.3) and Lemma 3.5]{Shparlinski2022_Gauss}. 
\begin{lem}\label{lem:squares_energy_general_improved}
    Let $F$ be irreducible. For any positive integer $m \leq r$ and any interval $\cI_m$ as in (\ref{eq:intervals}), 
    $$E_{F}^{\sq}(\cI_m) \leq q^{o(m)}\left(q^{4m-r} + q^{2m} \right)$$
\end{lem}
\begin{proof}
Recall that we are bounding the number of solutions to 
    \begin{align*}
    (x_1+s_0)^2 - (x_3+s_0)^2 \equiv (x_4+s_0)^2 - (x_2+s_0)^2 \Mod{F}
\end{align*}
with $\deg x_i < m$. By factoring and making the substitutions $y_1 = x_1-x_3$, $y_2 = x_1 + x_3$, $y_3 = x_4-x_2$ and $y_4 = x_4 + x_2$, we can equivalently count solutions to
\begin{align*}
    y_1(y_2 + 2s_0) \equiv y_3(y_4 + 2s_0) \Mod{F}, ~\deg y_i < m.
\end{align*}
Firstly, if $y_1 \equiv 0 \Mod{F}$ then either $y_3$ or $y_4$ is uniquely determined, so we can trivially obtain a bound of $O(q^{2m})$. We can obtain an identical bound if any of the other factors are equal to $0$. 

Therefore we now assume none of the factors are equivalent to $0$ modulo $F$. For any $a \in \F_q[T]$ let $J_{1,a}$ count the number of solutions to 
\begin{align}\label{eq:J1}
   y_1 \equiv ay_3 \Mod{F}, ~\deg y_i < m 
\end{align}
and let $J_{2,a}$ count the number of solutions to 
\begin{align}\label{eq:J2}
    (y_4 + 2s_0) \equiv a(y_2+2s_0)\Mod{F}, ~\deg y_i < m.
\end{align}
This allows us to write
\begin{align*}
   E_F^\sq(\cI_m) 
   \ll \sum_{\substack{\deg a < r \\ a \neq 0}}J_{1,a}J_{2,a}\ll \sum_{\substack{\deg a < r \\ a \neq 0}} J_{1,a}^2 + \sum_{\substack{\deg a < r \\ a \neq 0}}J_{2,a}^2.
\end{align*}
Now in the sum over $J_{2,a}$, we of course only need to consider $a$ such that $J_{2,a} > 0$. For such an $a$, fix some solution $(u,v)$ to \eqref{eq:J2}. Then for any other solution $(y_4, y_2)$ to \eqref{eq:J1} we have 
$$y_4-u \equiv a(y_2-v) \Mod{F}.$$
Of course $(y_4, y_2)$ uniquely determines $(y_4-u, y_2-v)$ which is a solution to \eqref{eq:J2}. Thus 
\begin{align*}
   E_F^\sq(\cI_m) 
   &\ll \sum_{\deg a < r}J_{1,a}^2\\
   &\ll \sum_{\deg a < r}H_{F,a}(m,m)^2
\end{align*}
with $H_{F,a}(m,m)$ as in \eqref{eq:def:hyperbola}. We can isolate $a=0$ and write
\begin{align*}
   E_F^\sq(\cI_m) 
   &\ll q^{2m} + \sum_{\substack{\deg a < r \\ a \neq 0} }H_{F,a}(m,m)^2\\
   &\ll q^{2m} + \max_{\substack{\deg b < r \\ b \neq 0}}H_{F,b}(m,m)\sum_{\substack{\deg a < r } }H_{F,a}(m,m).
\end{align*}
Applying Lemma \ref{lem:hyperbola} here yields
\begin{align*}
   E_F^\sq(\cI_m) 
   &\ll q^{o(m)}\left(q^{2m} + q^{4m-r}\right)
\end{align*}
as desired.
\end{proof}

Finally, for $E_{F }^{\sqrt}(\cI_m) $ as in (\ref{eq:def:energy_sqrt}), the following is given in \cite{BagshawKerr2023}. 
\begin{lem}\label{lem:sqrt_energy}
Assume $q$ is odd. Let $F$ be irreducible. For any integer $m \leq r$, 
$$E_{F}^{\sqrt}(m) \leq q^{o(m)}\left(q^{7m/2-r/2} + q^{2m }\right). $$
\end{lem}

\section{Proofs of main results}\label{sec:proofs}
In each of the proofs of our main results, we will let $S$ denote the sum in question. 
\subsection{Proof of Theorem \ref{thm:BilinearKloosterman1}}
For each $x$ we let $\gamma_x = e_F(s_0x + at_0\ov{x})$. Applying Corollary \ref{cor:charsum_subspace} multiple times yields 
\begin{align*}
    |S|
    &= \bigg{|}\sum_{\substack{\deg x < r \\ (x,F) = 1}}\gamma_x\sum_{\deg s < m}e_F(sx)\sum_{\deg t < n}e_{F}(at\ov{x})\bigg{|}\\
    &= q^n\bigg{|}\sum_{\substack{\deg x < r \\ (x,F) = 1 \\ \deg_{F}(a\ov{x}) <r-n}}\gamma_x\sum_{\deg s < m}e_F(sx)\bigg{|}\\
    &\leq q^{n+m}\sum_{\substack{\deg x < r-m \\ (x,F) = 1 \\ \deg_{F}(a\ov{x}) < r-n}}1.
\end{align*}
This sum over $x$ is bounded by $H_{F,a}(r-m, r-n)$, which completes the proof.

\subsection{Proof of Theorem \ref{thm:BilinearKloosterman2}}
For each $x$ we let $\gamma_x = e_F(at_0\ov{x})$.  Applying Corollary \ref{cor:charsum_subspace} we have 
\begin{align*}
   |S| \nonumber 
    &= \big{|}\sum_{s \in  \cS}\alpha_{s}\sum_{\substack{\deg x < r \\ (x,F) = 1}}e_F(sx)\gamma_x \sum_{\deg t < n} e_{F}(a\ov{x} t)\big{|} \nonumber \\
    &= q^n\big{|}\sum_{s \in \cS}\alpha_{s}\sum_{ \substack{\deg x < r \\ (x,F) = 1 \\ \deg_{F}(a\ov{x}) < r-n}}\gamma_xe_F(sx)\big{|}.
\end{align*}
 Applying the Cauchy-Schwarz inequality allows us to increase the sum over $s$ to $\deg s < r$, so we obtain 
\begin{align*}
     |S|^2 
     &\leq q^{2n}\|\balpha\|_2^2\sum_{\deg s < r }\Big{|}\sum_{\substack{\deg x < r \\ (x,F) = 1 \\ \deg_{F}(a\ov{x}) < r-n}} \gamma_xe_F(sx)\Big{|}^2\\
     &\leq q^{2n}\|\balpha\|_2^2\sum_{\substack{\deg x, \deg y < r \\ (x,F) = (y,F) = 1 \\ \deg_{F}(a\ov{x}), \deg_{F}(a\ov{y}) < r-n}}\Big{|}\sum_{\deg s < r} e_F(s(x-y))\Big{|}.
\end{align*}
By orthogonality, the inner sum is $0$ unless $x=y$, in which case it is equal to $q^r$. Thus
\begin{align*}
     |S|^2 \leq q^{2n+r}\|\balpha\|_2^2\sum_{\substack{\deg x < r \\ (x,F) = 1 \\ \deg_{F}(a\ov{x}) < r-n}}1.
\end{align*}
This sum over $x$ is bounded above by $H_{F,a}(r, r-n)$, which gives the desired result. 

\subsection{Proof of Theorem \ref{thm:BilinearKloosterman3}}
Here we set $F_0 = F/\gcd(a,F)$ and $a_0 = a/\gcd(a,F)$ and for each $x$ we let $\gamma_x = e_F(s_0x + at_0\ov{x})$.  We can apply Corollary \ref{cor:charsum_subspace} and the Cauchy-Schwarz inequality similarly to in the proof of Theorem \ref{thm:BilinearKloosterman2}, but we do not increase the sum over $s$. This then yields 
\begin{align*}
    |S|^2 
    &\leq q^{2n+m}\|\balpha\|_2^2\#\{(x,y) \in \F_q[T]^2 : \deg x, \deg y < r, \\
    &\hspace{9em}\deg_{F_0}(a_0\ov{x}), \deg_{F_0}(a_0\ov{y}) <  r-d-n \\
    &\hspace{14em}\text{ and } \deg(x-y) < r-m\}.
\end{align*}
For any $(x,y)$ in this set on the right, we have 
$$x - y \equiv h \Mod{F}$$
for some $\deg h < r-m$. Now suppose we fix $h$. Every $x$ and $y$ can be written uniquely in the form 
$$x = u_1 + F_0v_1, ~~y = u_2 + F_0v_2$$
for some $\deg u_i < r-d$ and $\deg v_i < d$. The conditions on $x$ and $y$ imply 
\begin{align*}
    u_1 + F_0v_1 - u_2 - F_0v_2 \equiv h \Mod{F}, ~~\deg_{F_0}(a_0\ov{u_i})< r-d-n.
\end{align*}
Of course any three of $(u_1, u_2, v_1, v_2)$ uniquely determines the other, and note $u_1 - u_2 \equiv h \Mod{F_0}$. Thus summing over $h$, and then making a change of variables $u_i \to a_0\ov{u_i}$ yields 
\begin{align*}
    |S|^2 &\leq q^{2n+m+d}\|\balpha\|_2^2\\
    &\:\:
\times \sum_{\deg h < r-m}\#\{(u_1,u_2) \in \F_q[T]^2 : \deg u_1, \deg u_2 < r-d, \\
    &\hspace{9em}\deg_{F_0}(a_0\ov{u_1}), \deg_{F_0}(a_0\ov{u_2}) <  r-d-n \\
    &\hspace{15em}\text{ and } u_1-u_2 \equiv h \Mod{F_0}\}\\
    &= q^{2n+m+d}\|\balpha\|_2^2A_{F_0, \ov{a_0}}(r-d-n, r-m),
\end{align*}
completing the proof.

\iffalse
\newpage
Second equation on page (23) is something like 
\begin{align}\label{to_email}
    W^2 
    &\ll \|\balpha\|_2^2\frac{Mq}{r}\sum_{1\leq|n|\leq q/M}J_r((ac)^{-1}n, K) + \textup{  extra stuff}
\end{align}
where $J_r((ac)^{-1}n, K)$ counts the solutions to 
$$x^{-1}-y^{-1} \equiv (ac)^{-1}n \Mod{r}, ~~x,y \in [1,K].$$
Now suppose that $\mathbf{q/M > r}$ (even suppose its much larger). Then $(ac)^{-1}n$ runs through all residue classes modulo $r$. So we should get something like 
\begin{align*}
\sum_{1\leq|n|\leq q/M}J_r((ac)^{-1}n, K)
&\leq \lfloor \frac{q/M}{r}\rfloor\sum_{1\leq|n|\leq r}J_r((ac)^{-1}n, K) + \textup{  extra stuff}\\
&\ll \frac{q/M}{r}K^2 + \textup{  extra stuff}.
\end{align*}
Substituting this into (\ref{to_email}) yields 
\begin{align*}
W^2
&\ll \|\balpha\|_2^2\frac{Mq}{r}\frac{qK^2}{Mr} + \textup{  extra stuff}\\
&= \|\balpha\|_2^2\frac{q^2K^2}{r^2} + \textup{  extra stuff}.
\end{align*}
But using $q/M > r$, this dominates the bounds given in the next line on page (23).

I think when applying it to bilinear kloosterman sums, its possible to have $M$
\newpage

\fi

\subsection{Proof of Theorem \ref{thm:BilinearGaussSum1}}
Again, for each $x$ we let $\gamma_x = e_F(at_0\ov{x})$. Rearranging the sum and applying Corollary \ref{cor:charsum_subspace} yields 
\begin{align*}
\left|S\right|
&\leq \Big{|}\sum_{s \in \cS}\alpha_{s}\sum_{\deg x 
< r}\gamma_xe_F(sx)\sum_{\deg t < n} e_F(tx^2)\Big{|}\\
&= q^n \Big{|}\sum_{s \in \cS}\alpha_{s}\sum_{\substack{\deg x < r \\ \deg_F(x^2) < r-n}}\gamma_xe_F(sx) \Big{|}.
\end{align*}
Now applying the H{\"o}lder inequality allows us to increase the sum over $s$ to $\deg s < r$, giving 
\begin{align*}
\left|S\right|^4
&\leq q^{4n}\|\balpha\|_1^2\|\balpha\|_2^2\sum_{\deg s < r }\Big{|}\sum_{\substack{\deg x < r \\ \deg_F(x^2) < r-n}}\gamma_xe_F(sx) \Big{|}^4\\
&\leq q^{4n}\|\balpha\|_1^2\|\balpha\|_2^2\sum_{\substack{\deg x < r \\ \deg_F(x^2) < r-n}}\Big{|}\sum_{\deg s < r }e_F(s(x_1 + x_2 - x_3 - x_4)) \Big{|}.
\end{align*}
By orthogonality, the inner sum is $0$ unless $x_1 + x_2 \equiv x_3 + x_4 \Mod{F}$, in which case it is equal to $q^r$. Thus 
\begin{align*}
|S|^4 
&\leq q^{4n + r}\|\balpha\|_1^2\|\balpha\|_2^2E_{F}^\sqrt(r-n)
\end{align*}
which completes the proof.

\subsection{Proof of Theorem \ref{thm:BilinearGaussSum2}}
Firstly by completing the square and applying Lemma \ref{lem:app:Gauss_1} we have

\begin{align}\label{eq:BGS2_initial}
|S|  
&= \Big{|}\sum_{s \in \cS}\sum_{\substack{t \in \cI_n \\ t \not\equiv 0(F)}}\alpha_s\beta_te_F(-s^2\ov{4at})G_F(0,at)\Big{|} \nonumber \\
&= q^{r/2}\Big{|}\sum_{s \in \cS}\sum_{\substack{t \in \cI_n \\ t \not\equiv 0(F)}}\alpha_s\beta_t\theta_te_F(-s^2\ov{4at})\Big{|}
\end{align}
for some $|\theta_t| =1$. Next, applying the H{\"o}lder inequality to (\ref{eq:BGS2_initial}) yields 
\begin{align*}
\left|S \right|^4
&\leq q^{2r}\|\balpha\|_1^2\|\balpha\|_2^2\sum_{s \in \cS}\Big{|}\sum_{\substack{t \in \cI_n \\ t \not\equiv 0(F)}}\beta_t\theta_te_F(-s^2\ov{4at}) \Big{|}^4.
\end{align*}
Since each element of $\cS$ is distinct modulo $F$, we can increase the sum on $s$ to a full set of residue classes modulo $F$ to give
\begin{align*}
\left|S\right|^4 
&\leq q^{2r}\|\balpha\|_1^2\|\balpha\|_2^2\sum_{\deg s < r}\Big{|}\sum_{\substack{t \in \cI_n \\ t \not\equiv 0(F)}}\beta_t\theta_te_F(-s^2\ov{4at}) \Big{|}^4 \\
&\ll q^{2r}\|\balpha\|_1^2\|\balpha\|_2^2\sum_{\deg s < r}\Big{|}\sum_{\substack{t \in \cI_n \\ t \not\equiv 0(F)}}\beta_t\theta_te_F(s\ov{t}) \Big{|}^4.
\end{align*}
We now expand the inner sum and apply orthogonality to give 
\begin{align*}
\left|S\right|^4 
&\ll q^{2r}\|\balpha\|_1^2\|\balpha\|_2^2\|\bbeta\|_\infty^4\sum_{\substack{(t_1,t_2,t_3,t_4) \in \cI_n \\ t_i \not\equiv 0(F)}}\sum_{\deg s < r}e_F(s(\ov{t_1} + \ov{t_2} - \ov{t_3} - \ov{t_4}))\\
&\ll q^{3r}\|\balpha\|_1^2\|\balpha\|_2^2\|\bbeta\|_\infty^4E_F^{\inv}(\cI_n),
\end{align*}
as desired. 
\subsection{Proof of Theorem \ref{thm:BilinearGaussSum3}}
Identically to the start of the proof of Theorem \ref{thm:BilinearGaussSum3}, 

\begin{align}\label{eq:BGS3_initial}
|S|  
&= q^{r/2}\Big{|}\sum_{s \in \cI_m}\sum_{\substack{t \in \cI_n \\ t \not\equiv 0(F)}}\alpha_s\beta_t\theta_te_F(-s^2\ov{4at})\Big{|}
\end{align}
for some $|\theta_t| =1$. Next applying the Cauchy-Schwarz inequality to (\ref{eq:BGS3_initial}), and changing the order of summation, yields 
\begin{align*}
    |S|^2
    &\leq q^r\|\balpha\|_2^2\sum_{s \in \cI_m}\sum_{\substack{(t_1,t_2) \in \cI_n^2 \\ t_i \not\equiv 0(F)}}\beta_{t_1}\ov{\beta}_{t_2}e_F(-s^2(\ov{4at_1} - \ov{4at_2}))\\
    &\leq q^r\|\balpha\|_2^2\sum_{\substack{(t_1,t_2) \in \cI_n^2 \\ t_i \not\equiv 0(F)}}\beta_{t_1}\ov{\beta}_{t_2}\sum_{s \in \cI_m}e_F(-s^2(\ov{4at_1} - \ov{4at_2})).
\end{align*}
For any $k \in \F_q[T]$ let $I_{F,k}^{-}(\cI_n)$ denote the number of solutions to 
$$\ov{t_1} - \ov{t_2} \equiv k \Mod{F}, ~t_i \in \cI_n. $$
Then we can write 
\begin{align*}
    |S|^2
    &\leq q^r\|\balpha\|_2^2\sum_{\deg k < r}I_{F, k}^{-}(\cI_n)\Big{|}\sum_{s \in \cI_m}e_F(-s^2\ov{4a}k)\Big{|}.
\end{align*}
Of course 
$$I_{F,k}^{-}(\cI_n) = I_{F,k}^{-}(\cI_n)^{1/2}(I_{F,k}^{-}(\cI_n)^2)^{1/4}$$
so now applying the H{\"o}lder inequality we obtain 
\begin{align*}
    |S|^8 &\leq q^{4r}\|\balpha\|_2^8\|\bbeta\|_\infty^8 \left(\sum_{\deg k < r}I_{F,k}^{-}(\cI_n)\right)^2\\
    &\qquad\qquad \times \sum_{\deg k < r}I_{F,k}^{-}(\cI_n)^2\sum_{\deg k < r}\left|\sum_{s \in \cI_m}e_F(-s^2k\ov{4a}) \right|^4.
\end{align*}
By construction we have 
$$\sum_{\deg k < r}I_{F,k}^{-}(\cI_n) \leq q^{2n}$$
and 
$$\sum_{\deg k < r}I_{F,k}^{-}(\cI_n)^2 = E_F^{\inv}(\cI_n)$$
so substituting gives 
$$|S|^8 \leq q^{4r + 4n}\|\balpha\|_2^8\|\bbeta\|_\infty^8 E_F^{\inv}(\cI_n)\sum_{\deg k < r}\Big{|}\sum_{s \in \cI_m}e_F(-s^2k\ov{4a}) \Big{|}^4.$$
Finally by making a change of variables, expanding the inner sum, interchanging summation and applying orthogonality we can conclude 
\begin{align*}
|S|^8 
&\leq q^{4r + 4n}\|\balpha\|_2^8\|\bbeta\|_\infty^8 E_F^{\inv}(\cI_n)\sum_{\deg k < r}\Big{|}\sum_{s \in \cI_m}e_F(s^2k)\Big{|}^4\\
&= q^{4r + 4n}\|\balpha\|_2^8\|\bbeta\|_\infty^8 E_F^{\inv}(\cI_n) \\ &\qquad \times \sum_{(s_1,s_2,s_3,s_4) \in \cI_m}\sum_{\deg k < r}e_F(k(s_1^2 + s_2^2 - s_3^2 - s_4^2))\\
&= q^{5r + 4n}\|\balpha\|_2^8\|\bbeta\|_\infty^8 E_F^{\inv}(\cI_n)E_F^{\sq}(\cI_m).
\end{align*}

\section{Acknowledgements}
The author would like to thank Igor Shparlinski and Bryce Kerr for many helpful comments and suggestions, and for reading over multiple drafts of this paper. In particular, the author would like to thank Bryce Kerr for the idea for the proof of Lemma \ref{lem:squares_energy_general_improved}. 

During the preparation of this work, the author was supported by an Australian Government Research Training Program (RTP) Scholarship. 

\bibliographystyle{plain} 

\bibliography{refs}

\appendix

\section{Kloosterman and Gauss sums in function fields}\label{appendix:kloos_and_gauss}
Throughout this paper, we have used many general and basic properties of Kloosterman and Gauss sums in $\F_q[T]$, as well as a few more specialized properties (such as in the proof of Lemma \ref{lem:inverse_via_explicit_kloos}). The following appendix contains proofs (or, references to proofs) for the properties used. We have included them in an appendix, as many of them are standard for Kloosterman and Gauss sums over $\mathbb{Z}$ and the proofs mostly consist of technical details. We also recognize that many of them are known in this setting also, but if we have been unable to find a proof in the literature then we have included it here. Whenever possible we have attempted to simplify the proofs, and we have heavily relied on well-known ideas from \cite{E1961, Iwaniec1997, IwaniecKowalski2004} as well as newer ideas from \cite{KerrShparlinskiWuXi2022}. 

Again, we fix $F \in \F_q[T]$ and $\deg F = r$. Although many of the results in this section require $q$ to be odd, if that condition is not explicitly stated then it can be assumed the result holds for arbitrary $q$. 

\subsection{Gauss sums}
We first will discuss some basic properties of quadratic Gauss sums in function fields. Recall that for any $s,t \in \F_q[T]$ we define the quadratic Gauss sum 
$$G_F(s,t) = \sum_{\deg x < r}e_F(sx + tx^2).$$
This first result is given in \cite[Lemma 6.5]{BaierSingh2022}. 
\begin{lem}\label{lem:app:Gauss_1} 
Assume $q$ is odd. For any $s,t \in \F_q[T]$ if  $\gcd(t,F) = 1$ then 
$$|G_F(s,t)| = q^{r/2}.$$. 

\end{lem}
We can also show the following. 
\begin{lem}\label{lem:app:Gauss_2}
    For any $s,t \in \F_q[T]$ if  $\gcd(t,F) = D$ then
    $$
    G_F(s,t) = 
    \begin{cases}
        q^{\deg D}G_{F/D}(s/D, t/D), & D|s\\
        0, &D \nmid s.
    \end{cases}
    $$
\end{lem}
\begin{proof}
First, every $\deg x < r$ can be written uniquely as $u+v(F/D)$ for some $\deg u < r-\deg D$ and $\deg v < \deg D$. Thus  
    \begin{align*}
        G_F(s,t)
        &= \sum_{\deg u < r-\deg D}\sum_{\deg v < \deg D}e\left(\frac{t(u + vF/D)^2 + s(u+vF/D)}{F} \right)\\
        &= \sum_{\deg u < r-\deg D}e\left(\frac{tu^2 + su }{F} \right)\sum_{\deg v < \deg D}e\left(\frac{sv}{D} \right).
    \end{align*}
The inner sum is $0$ if $s \not\equiv 0 \Mod{D}$, otherwise the sum is equal to $q^{\deg D}$, and we write
 \begin{align*}
        G_F(s,t)
        &= q^{\deg D}\sum_{\deg u < r-\deg D}e\left(\frac{(t/D)u^2 + (s/D)u }{F/D} \right)
    \end{align*}
    as desired. 
\end{proof}

We let $\legendre{~\cdot~}{~\cdot~}_q$ denote the Legendre-Jacobi symbol in $\F_q[T]$.  The first assertion in the following can be found in \cite[Theorem 4.1]{Hsu2003} and the second in \cite[Lemma 6.4]{BS2022}. 
\begin{lem}\label{lem:Hsu_Gauss} 
    Assume $q$ is odd, and let $q = p^\ell$ for some prime $p$ and positive integer $\ell$. Then 
    $$G_F(0, 1) = q^{r/2}\epsilon_F$$
    where
    \begin{align}\label{eq:app:epsilon_F}
        \epsilon_F = 
    \begin{cases}
    1, & 2|r,\\
    -\chi_q(\textup{sgn}(F))(-1)^\ell, &2\nmid r \textup{ and } p \equiv 1 \Mod{4},\\
    -\chi_q(\textup{sgn}(F))(-i)^{\ell} &2\nmid r \textup{ and } p \equiv 3 \Mod{4},
    \end{cases}
    \end{align}
    $\textup{sgn}(F)$ is the leading coefficient of $F$ and $\chi_q$ is the quadratic character of $\F_q$. 

    Furthermore, if $F$ is irreducible then for any $t \in \F_q[T]$ with $\gcd(t,F) = 1$ we have
    $$G_F(0,t) = \legendre{t}{F}_qG_F(0,1).$$
\end{lem}

We next consider Gauss sums over reduced residue classes
$$ G^*_F(s,t) = \sum_{\substack{\deg x < r \\ (x,F) = 1}}e_F(sx + tx^2).$$
Using the previous results, we can prove the following. 
\begin{lem}\label{lem:gauss_sum_reduced_residues}
Assume $q$ is odd.  For any $s,t \in \F_q[T]$, 
    $$\left|G^*_F(s,t) \right|\leq q^{r/2 + \deg(s, t, F)/2 + o(r)}.$$
\end{lem}
\begin{proof}
It is convenient to introduce an $\F_q[T]$ analogue of the classical M{\"o}bius function, 
\begin{align*}
\mu_q(x) 
=
\begin{cases}
(-1)^k, &x \text{ is square-free and a product of $k$ distinct}\\& \text{irreducible factors,}\\
0, &\text{otherwise}.
\end{cases}
\end{align*} 
Just as for the M{\"o}bius function $\mu$, we have that for any $x \in \F_q[T]$, $\mu_q$ satisfies 
$$\sum_{\substack{D | x \\ D \textup{ monic}}}\mu_q(D) = 
\begin{cases}
1, ~&x \textup{ is constant},\\
0, &\textup{otherwise}. 
\end{cases}$$
Thus manipulating in a standard way we write 
    \begin{align*}
        \left|G^*_F(s,t)\right|
        &= \Big{|}\sum_{\substack{\deg x < r }}e_F(tx^2 + sx) \sum_{\substack{D | (x,F) \\ D \textup{ monic}}}\mu_q(D) \big{|}\\
        &= \Big{|}\sum_{\substack{D|F \\ D\textup{ monic}}}\mu_q(D)\sum_{\substack{\deg x < r \\ D|x}}e_F(tx^2 + sx) \Big{|}\\
        &= \Big{|}\sum_{\substack{D|F \\ D\textup{ monic}}}\mu_q(D)G_{F/D}(s, tD)\Big{|}.
    \end{align*}
By Lemma \ref{lem:app:Gauss_2} the inner Gauss sum is $0$ unless $\gcd(tD, F/D) |s$, in which case we can apply Lemma \ref{lem:app:Gauss_1} to see
\begin{align*}
    |G^*_F(s,t)|
    &\leq \sum_{\substack{D|F \\ D\textup{ monic} \\ (F/D, tD)|s}}q^{\deg(F/D, tD) + (r-\deg D - \deg(F/D, tD))/2}\\
    &\leq q^{r/2}\sum_{\substack{D|F \\ D\textup{ monic} \\ (F/D, tD)|s}}q^{\deg(t, F/D)/2}\\
    &= q^{r/2}\sum_{\substack{D|F \\ D\textup{ monic} \\ (F/D, tD)|s}}q^{\deg(s, F/D, t)/2}\\
    &\leq q^{r/2 + \deg(s, F, t)/2 + o(r)}
\end{align*}
where the final line comes from applying Lemma \ref{lem:divisor_bound}. 
\end{proof}

In the case that $t \equiv 0 \Mod{F}$, $G_{F}^*(s,t)$ reduces to the Ramanujan sum 
$$C_F(s) = \sum_{\substack{\deg x < r \\ (x,F) = 1}}e_F(xs).$$
We can give a simple bound as follows. 
\begin{lem}\label{lem:ramanujan}
For any $s \in \F_q[T]$, 
$$C_F(s)  \leq 2q^{\deg(F,s)}. $$
\end{lem}
\begin{proof}
Firstly, if we let $g = \gcd(s,F)$ and $F_0 = F/g$ then
since $e_F$ is periodic modulo $F$, we can manipulate in a standard way to see
\begin{align*}
    C_F(s) 
    &= \sum_{\substack{\deg x < r \\ (x,F) = 1}}e_F(xg)
    = \sum_{\substack{\deg x < r \\ (x,F) = 1}}e_{F_0}(x) \\
    &= \sum_{\substack{\deg x < r \\ (x,F_0) = 1}}e_{F_0}(x) - \sum_{\substack{\deg x < r-\deg g}}e_{F_0}(xg).
\end{align*}
The second sum is $0$, unless $F_0|g$. Either way, its bounded above by $q^{\deg g}$. Since $e_{F_0}$ is periodic modulo $F_0$ we can deal with the first sum as 
\begin{align*}
    \left|C_F(s) \right|
    &\leq q^{\deg g}\Big{|}\sum_{\substack{\deg x < r-\deg g \\ (x,F_0) = 1}}e_{F_0}(x)\Big{|} + q^{\deg g}.
\end{align*}
We can manipulate this sum almost identically to in the proof of Lemma \ref{lem:gauss_sum_reduced_residues}, giving 
\begin{align*}
    \left|C_F(s) \right|
    &\leq q^{\deg g}\Big{|}\sum_{\substack{D|(F_0, 1) \\ D \textup{ monic}}}q^{\deg D}\mu_q(F/D)\Big{|} + q^{\deg g}
    \leq 2q^{\deg g}.
\end{align*}
\end{proof}

\subsection{Kloosterman sums}
We can now similarly discuss properties of Kloosterman sums in function fields. Recall that for any $s,t \in \F_q[T]$ we have defined the Kloosterman sum 
$$K_F(s,t) = \sum_{\substack{\deg x < r \\ (x,F) = 1}}e_F(sx+ t\ov{x}).$$
Firstly, we remark that Kloosterman sums trivially satisfy the identity 
\begin{align}\label{eq:app:kloos_move_over_gcd_1}
    K_F(s,ct) = K_F(cs,t)
\end{align}
if $\gcd(c,F) = 1$. We will also make repeated use (often without reference) to the fact that 
$$K_F(s,t) = K_F(t,s).$$
These sums also satisfy twisted multiplicativity as demonstrated below. 

\begin{lem}\label{lem:app:twisted_mult}
Suppose $F = F_1F_2$ with $\gcd(F_1, F_2) = 1$. Then 
$$K_F(s,t) = K_{F_1}(s\ov{F}_2, t\ov{F}_2)K_{F_2}(s\ov{F}_1, t\ov{F}_1) $$
where $\ov{F}_1$ (resp. $\ov{F}_2$) denotes the multiplicative inverse of $F_1 \Mod{F_2}$ (resp $F_2 \Mod{F_1}$).
\end{lem}
\begin{proof}

Expanding out $K_{F_1}(s\ov{F}_2, t\ov{F}_2)K_{F_2}(s\ov{F}_1, t\ov{F}_1)$ gives 
\begin{align*}
\sum_{\substack{x < \deg F_1\\ (x,F_1) = 1}}\sum_{\substack{y  < \deg F_2\\ (y,F_2) = 1}} e_F\left({s(F_2\ov{F}_2x + F_1\ov{F}_1y) + t(F_2\ov{F}_2\ov{x} + F_1\ov{F}_1\ov{y})} \right).
\end{align*}
Under this notation $x\ov{x} \equiv 1 \Mod{F_1}$ and $y\ov{y} \equiv 1 \Mod{F_2}$. The Chinese remainder theorem tells us that $F_2\ov{F}_2x + F_1\ov{F}_1y$ runs through all reduced residue classes modulo $F$. Thus, it now suffices to show that 
\begin{align}\label{eq:chinese_kloost_inverse}
    (F_2\ov{F}_2x + F_1\ov{F}_1y)(F_2\ov{F}_2\ov{x} + F_1\ov{F}_1\ov{y}) \equiv 1 \Mod{F}. 
\end{align}

We write $x\ov{x} = d_1F_1 +1$ and $y\ov{y} = d_2F_2 + 1$ for some $d_1, d_2 \in \F_q[T]$. Since $e_{F_i}$ is periodic modulo $F_i$, we are free to choose suitable $\ov{F}_1, \ov{F}_2$. We choose $\ov{F}_1$ and $\ov{F}_2$ such that $F_1\ov{F}_1 + F_2\ov{F}_2 = 1$. Substituting these identities into (\ref{eq:chinese_kloost_inverse}) finishes the proof.
\end{proof}

In the next part of this section, as we move towards the proof of Lemma \ref{lem:lem:explicit_kloosterman_primepower}, we follow very closely the ideas presented in \cite[Chapter 4]{Iwaniec1997}. 
\begin{lem}\label{lem:app:kloos_explicit_prime_power_even} 
Assume $q$ is odd. If $F = P^{2j}$ for some irreducible $P$ and some $j \geq 1$, and $\gcd(P, s) = 1$ then 
    $$K_F(s,s) = 2q^{r/2}\textup{Re}~e_F(2s).$$
\end{lem}
\begin{proof}
Every $x$ with $\deg x < \deg F$ and $\gcd(x,F) =1$ can be written as 
$$x \equiv u(1+vP^{j}) \Mod{F}$$
for some $\deg v < \deg P^{j}$ and $\deg u < \deg P^{2j}$ and $\gcd(u,P) = 1$. For each choice of $v$, there exists a unique $u$ that satisfies this congruence. Thus, any given $x$ can be written like this in $q^{j \deg P}$ ways. Also, note that a symbolic computation yields
$\ov{u(1+vP^{j})} \equiv \ov{u}(1-vP^{j}) \Mod{F}.$ Thus 
    \begin{align*}
K_F(s,s) 
&= q^{-j\deg P}\sum_{\substack{\deg u < \deg P^{2j} \\ \gcd(u,P) = 1}}\sum_{\deg v < \deg P^{j}}e\left( \frac{su + suvP^{j} + s\ov{u} - s\ov{u}vP^{j}}{P^{2j}}\right)\\
&= \sum_{\substack{\deg u < \deg P^{2j} \\ \gcd(u,P) = 1 \\ u \equiv \ov{u} ({P^{j}})}}e\left( s\frac{u +\ov{u}}{P^{2j}}\right)
\end{align*}
where the second line follows by orthogonality. The congruence condition on $u$ in the sum implies $u^2 \equiv 1 \Mod{P^{j}}$ and thus $u$ satisfies $u = \pm 1 + tP^{j}$ for some $\deg t < j\deg P$. Also note that $\ov{\pm 1 + tP^{j}} \equiv \pm1 - tP^{j} \Mod{F}$. Substituting this, and simplifying yields 
\begin{align*}
K_F(s,s) 
&= \sum_{\deg t < j\deg P}\left({e\left( \frac{2s}{P^{2j}}\right) + e\left( \frac{-2s}{P^{2j}}\right)} \right)\\
&= 2q^{j \deg P}\text{Re}~e_F(2s).
    \end{align*}
\end{proof}

    \begin{lem}\label{lem:app:kloos_explicit_prime_power_odd} 
Assume $q$ is odd. If $F = P^{2j + 1}$ for some irreducible $P$ and some $j \geq 1$, and $\gcd(P, s) = 1$ then 
    $$K_F(s,s) = 2\legendre{s}{P}_q q^{ \deg P(j + 1/2)}\textup{Re} ~e_F(2s) \epsilon_F$$
where $\epsilon_F$ is as in (\ref{eq:app:epsilon_F}). 
\end{lem}
\begin{proof}
Similar to before, every $x$ with $\deg x < \deg F$ and $\gcd(x,F) = 1$ can be written as 
$$x \equiv u(1+vP^{j + 1}) \Mod{F}$$
for some $\deg v < \deg P^{j}$ and $\deg u < \deg P^{2j + 1}$ with $\gcd(u,P) = 1$, in exactly $q^{j \deg P}$ ways. Thus, as in the proof of Lemma \ref{lem:app:kloos_explicit_prime_power_even},
\begin{align*}
K_F(s,s) 
&= \sum_{\substack{\deg u < \deg P^{2j + 1} \\ \gcd(u,P) = 1 \\ u \equiv \ov{u} \Mod{P^{j}}}}e\left( s\frac{u +\ov{u}}{P^{2j + 1}}\right).
\end{align*}
Again, the solutions to this congruence are $u = \pm 1 + tP^{j}$ with $\deg t < \deg P^{j + 1}$, and note that $\ov{\pm 1 + tP^{j}} \equiv \pm1 - tP^{j} \pm t^2P^{2j} \Mod{F}.$ Thus 
\begin{align*}
K_F(s,s) 
&= \sum_{\deg t < \deg P^{j + 1}}\left({e\left( s\frac{2 +t^2P^{2j}}{P^{2j + 1}}\right) + e\left( s\frac{-2-t^2P^{2j}}{P^{2j + 1}}\right)}\right)\\
&= 2\text{Re} \sum_{\deg t < \deg P^{j + 1}}{e\left( s\frac{2 +t^2P^{2j}}{P^{2j + 1}}\right)}\\
&= 2q^{j \deg P}\textup{Re}~e_F(2s) G_P(0,s).
\end{align*}
Thus, it suffices to show
   $$ G_P(0,s) = \legendre{s}{P}_q q^{\deg P/2}\epsilon_P.$$
By Lemma \ref{lem:Hsu_Gauss}, 
$$ \sum_{\deg t < \deg P}{e\left(\frac{st^2}{P}\right)} = \legendre{s}{P}_qG_F(0,1)$$
and now applying Lemma \ref{lem:Hsu_Gauss} finishes the proof, after noting that $\epsilon_P = \epsilon_F$ since $F$ is equal to $P$ raised to an odd power. 

\end{proof}

We can now combine the previous results in the following Lemma. 
\begin{lem}\label{lem:lem:explicit_kloosterman_primepower}
Assume $q$ is odd. Suppose $F = P^j$ for some irreducible $P$ and some $j \geq 2$. Suppose $s,t \in \F_q[T]$ satisfy $\gcd(F, st) = 1$. Then $K_F(s,t) = 0$ unless there exists some $c \in \F_q[T]$ such that $s \equiv c^2t \Mod{F}$, in which case 
    $$K_F(s,t) = 2\legendre{ct}{F}_qq^{r/2} \textup{Re} ~e_F(2ct) \epsilon_F$$
    where $\epsilon_F$ is as in (\ref{eq:app:epsilon_F}). 
\end{lem}
\begin{proof}
    Every $x$ satisfying $\deg x < \deg F$ can be written uniquely as 
    $$x = u + P^{j-1}v$$
    for some $\deg u < \deg P^{j-1}$ with $(u,P) = 1$ and $\deg v < \deg P$. Also, note that
    $$\ov{u + P^{j-1}v} \equiv \ov{u} - P^{j-1}\ov{u}^2v \Mod{F}.$$
    Thus 
    \begin{align}\label{eq:Kloosterman_basic_split}
        K_F(s,t)
        &= \sum_{\substack{\deg u < \deg P^{j-1} \\ (u,P) =1}}\sum_{\deg v < \deg{P}}e\left( \frac{s(u+P^{j-1}v)+t(\ov{u} - P^{j-1}\ov{u}^2v)}{P^j}\right) \nonumber \\
        &= \sum_{\substack{\deg u < \deg P^{j-1} \\ (u,P) = 1}}e\left( \frac{su + t\ov{u}}{P^j}\right)\sum_{\deg v < \deg P}e\left(v\frac{s-t\ov{u}^2}{P} \right)
    \end{align}
    By orthogonality, the inner sum vanishes unless $s \equiv t\ov{u}^2 \Mod{P}$. In this case, since $s,t$ are coprime to $P$ this implies $s \equiv tc^2 \Mod{F}$ since $F$ is a power of $P$, for some $c$ satisfying $\gcd(c,P) = 1$. Thus by  (\ref{eq:app:kloos_move_over_gcd_1})
    $$K_{F}(s,t) = K_{F}(c^2t, t) = K_{F}(ct, ct) $$
    and the result follows by Lemmas~\ref{lem:app:kloos_explicit_prime_power_odd} and \ref{lem:app:kloos_explicit_prime_power_even}. 

\end{proof}

We can use equation (\ref{eq:Kloosterman_basic_split}) from the previous proof to give simple proofs of the next two results. 
\begin{lem}\label{lem:kloos=0_if_gcd=1_and_other_divisible}
        Suppose $F = P^j$ for some irreducible $P$ and some $j \geq 2$. For any $s,t \in \F_q[T]$ such that $P|s$ and $(t,P) = 1$ we have 
        $$K_F(s,t) = 0. $$
\end{lem}
\begin{proof}
    Substituting $P|s$ into the inner sum of (\ref{eq:Kloosterman_basic_split}) gives 
    \begin{align*}
        K_F(s,t)
        &= \sum_{\substack{\deg u < \deg P^{j-1} \\ (u,P) = 1}}e\left( \frac{su + t\ov{u}}{P^j}\right)\sum_{\deg v < \deg P}e\left(v\frac{-t\ov{u}^2}{P} \right).
    \end{align*}
    The inner sum is always $0$, by orthogonality. 
\end{proof}

    \begin{lem}\label{lem:kloos_divide_by_gcd}
        Suppose $F = P^j$ for some irreducible $P$ and some $j \geq 2$. For $s,t \in \F_q[T]$, suppose that $P^k|s$ and $P^k|t$ for some integer $k < j$. Then  
        
        $$K_F(s,t) = q^{k \deg P}K_{F/P^k}(s/P^k, t/P^k).$$
\end{lem}
\begin{proof}
    If $k = 0$ then the statement is immediate, so we may assume $1 \leq k < j$. We proceed by induction. If $k = 1$, then substituting into (\ref{eq:Kloosterman_basic_split}) and applying orthogonality immediately gives 
    $$K_F(s,t) = q^{\deg P}K_{F/P}(s/P, t/P). $$
    Now suppose the result holds for some integer $k \leq j-2$, but suppose $P^{k+1}|s$ and $P^{k+1}|t$. Then by the induction hypothesis  
    \begin{align*}
        K_F(s,t) 
        &= q^{k\deg P}K_{F/P^{k}}(s/P^{k}, t/P^{k}).
    \end{align*}
    Now we still have $P|s/P^K$ and $P|t/P^k$ so applying (\ref{eq:Kloosterman_basic_split}) identically as in the base case we conclude
        \begin{align*}
        K_F(s,t) 
        &= q^{(k+1)\deg P}K_{F/P^{k+1}}(s/P^{k+1}, t/P^{k+1}).
    \end{align*}
\end{proof}
Combining the previous two results now implies the following Lemma. 
    \begin{lem}\label{lem:kloos=0_gcd_not_equal}
        Suppose $F = P^j$ for some irreducible $P$ and some $j \geq 2$. For $s,t \in \F_q[T]$, suppose that $\gcd(s,F) \neq \gcd(t,F)$, and $P^{j-1} \nmid \gcd(s,F)$ or $P^{j-1}\nmid \gcd(t,F)$. Then  
        $$K_F(s,t) = 0.$$
\end{lem}
\begin{proof}
    Without loss of generality, we may suppose that $\gcd(t,F) = P^{k_1}$ and $\gcd(s,F) = P^{k_2}$ for some integers $k_1 < k_2$ and $k_1 \leq j-2$. Then by Lemma \ref{lem:kloos_divide_by_gcd}, 
    $$K_F(s,t) = q^{k_1\deg P}K_{F/P^{k_1}}(s/P^{k_1}, t/P^{k_1}).$$
    Now we know $\gcd(t/P^{k_1}, F/P^{k_1}) = 1$, and since $P^2 |F/P^{k_1}$ and $P | s/P^{k_1}$, we can apply Lemma \ref{lem:kloos=0_if_gcd=1_and_other_divisible}  which completes the proof. 
\end{proof}

We can now present the $\F_q[T]$ analogue of the classical Weil-Estermann bound for Kloosterman sums over $\Z$. 
\begin{lem}\label{lem:weil_bound} Assume $q$ is odd. For any $s,t \in \F_q[T]$, 
$$\left|K_F(s,t)\right| \leq 2^{\omega(F)}q^{r/2 + \deg(s,t,F)/2} $$
where $\omega(F)$ is the number of distinct, irreducible, monic divisors of $F$. 
\end{lem}
\begin{proof}
    By multiplcativity in Lemma (\ref{lem:app:twisted_mult}), it suffices to show this in the case of $F = P^j$ for some irreducible $P$ and $j \geq 1$. 
    
    If $F|s$ and $F|t$ then the sum is trivially bounded by $q^r$. If $F|s$ and $F \nmid t$, then this reduces to $C_F(t)$, and by Lemma \ref{lem:ramanujan} this is bounded by 
    $$C_F(t) \ll q^{\deg(F,t)} \leq q^{r/2 + \deg(F,t,s)/2}$$
    as desired. Of course the case $F\nmid s$ and $F|t$ is identical.

    Now consider the case $F \nmid s$ and $F\nmid t$. If $j=1$, then the result follows immediately from the Weil bound (for example, see \cite{weil1948}). If $j \geq 2$, then let $P^k = \gcd(s,t,F)$. By Lemma \ref{lem:kloos_divide_by_gcd} we have 
    $$\left|K_F(s,t)\right| = q^{\deg(s,t,F)}K_{F/P^k}(s/P^k, t/P^k).$$
    If $F/P^k = P$ then the result follows from the Weil bound. Otherwise, if now $\gcd(F/P^k, st/P^{2k}) = 1$ then the result follows by applying Lemma \ref{lem:lem:explicit_kloosterman_primepower}, or the sum is $0$ by Lemma \ref{lem:kloos=0_if_gcd=1_and_other_divisible}. 
\end{proof}

For most purposes, we will apply Lemma \ref{lem:divisor_bound} to Lemma \ref{lem:weil_bound}
and simply use 
$$\left|K_F(s,t)\right| \leq q^{r/2 + \deg(s,t,F)/2 + o(r)}.$$

We finally turn our attention to bounding a rather specific sum of Kloosterman sums, which appears in our Lemma \ref{lem:inverse_via_explicit_kloos}. 
\begin{lem}\label{lem:T_bound_final}
Assume $q$ is odd. For any $u,v,a \in \F_q[T]$, 
\begin{align*}
    &\left|\sum_{\deg t < r}K_F(u,t)K_F(v,t)e_F(-at)\right|\\ 
    &\hspace{8em}\leq q^{{3r}/{2} + {\deg(u-v,a,F/(u,v,F))/2} + {\deg(u,v,F)/2} + o(r)}.
\end{align*}
\end{lem}
\begin{proof}
    We will call the sum in question $T_{F,a}(u,v)$. We will deal with the case that $F$ is a power of an irreducible polynomial, and then the result follows by multiplicativity (Lemma \ref{lem:app:twisted_mult}). Thus we write $F = P^j$ for some irreducible polynomial $P$ and positive integer $j$, and consider four cases: 
    
\textit{Case I: $j=1$ and $F | a$}.  By rearranging and applying orthogonality, 
\begin{align*}
    \left|T_{F,a}(u,v)\right|
    &= \Big{|}\sum_{\substack{(x_1, x_2) \in \F_q[T]^2 \\ \deg x_i < r \\ x_i \neq 0}}e_F(ux_1 + vx_2)\sum_{\deg t < r }e_F(t(\ov{x_1} + \ov{x_2}))\Big{|}\\
    &= q^{r}\bigg{|}\sum_{\substack{(x_1, x_2) \in \F_q[T]^2 \\ \deg x_i < r \\ x_i \neq 0 \\ x_1 \equiv -x_2 (F)}}e_F(ux_1 + vx_2)\bigg{|}\\
    &= q^r\left|C_F(u-v)\right|.
\end{align*}
Thus by Lemma \ref{lem:ramanujan} we have 
\begin{align}\label{eq:T_bound_caseI}
    \left|T_{F,a}(u,v)\right| \ll q^{r + \deg(F, u-v)}
\end{align}
which is stronger than the desired result, since $F$ is irreducible. 

\textit{Case II: $j=1$ and $F \nmid a$}. Again expanding and applying orthogonality, 
\begin{align*}
    \left|T_{F,a}(u,v)\right|
    &= \Big{|}\sum_{\substack{(x_1, x_2) \in \F_q[T]^2 \\ \deg x_i < r \\ x_i \neq 0}}e_F(ux_1 + vx_2)\sum_{\deg t < r }e_F(t(\ov{x_1} + \ov{x_2} - a))\Big{|}\\
    &= q^{r}\Big{|}\sum_{\substack{(x_1, x_2) \in \F_q[T]^2 \\ \deg x_i < r \\ x_i \neq 0 \\ \ov{x_1} + \ov{x_2} \equiv a (F)}}e_F(ux_1 + vx_2)\Big{|}.
\end{align*}
We note that $\ov{x_1} + \ov{x_2} \equiv a \Mod{F}$, together with $x_i \neq 0$ implies that $x_i \not\equiv \ov{a} \Mod{F}$, implying $ax_1 \not\equiv 1 \Mod{F}$. Thus for a given $x_1$, by inspection one sees that ${x_2} \equiv x_1\ov{(ax_1 - 1)} \Mod{F}$ is the unique solution to $\ov{x_1} + \ov{x_2} \equiv a \Mod{F}$. Making this substitution gives  
\begin{align*}
    \left|T_{F,a}(u,v)\right|
    &= q^{r}\Big{|}\sum_{\substack{\deg x < r \\ x \neq 0, \ov{a} }}e_F(ux + vx\ov{(ax - 1)})\Big{|}.
\end{align*}
We now make the change of variables $x \to \ov{a}(x + 1)$ and substitute to give 
\begin{align}\label{eq:T_bound_caseII}
    \left|T_{F,a}(u,v)\right|
    &= q^{r}\Big{|}\sum_{\substack{\deg x < r \\ x \neq 0, -1 }}e_F(u\ov{a}x +v\ov{a}\ov{x})e_F(u\ov{a} + v\ov{a})\Big{|} \nonumber \\
    &= q^{r}\Big{|}e_F(u\ov{a} + v\ov{a})K_F(u\ov{a}, v\ov{a})-1 \Big{|}\nonumber \\
    &\ll q^{3r /2 + \deg(F, u,v)/2},
\end{align}
where the final line follows from Lemma \ref{lem:weil_bound}.

\textit{Case III: $j \geq 2$, and $F|u$ or $F|v$}. Without loss of generality, we only consider $F|u$. In this case, we rearrange and apply Lemma \ref{lem:ramanujan},
\begin{align*}
    \left|T_{F,a}(u,v)\right|
    &= \Big{|}\sum_{\deg t < r}C_F(t)K_{F}(v,t)e_F(-at)\Big{|}\\
    &\ll \sum_{\deg t < r}q^{\deg(F,t)}|K_{F}(v,t)|.
\end{align*}
Next we apply Lemma \ref{lem:weil_bound} and also (\ref{eq:sum_gcd}) to conclude 
\begin{align}\label{eq:T_bound_caseIII}
        \left|T_{F,a}(u,v)\right|
        &\ll \sum_{\deg t < r}q^{\deg(F,t) + \deg(F,v,t)/2 + r/2 + o(r)} \nonumber\\
        &\ll  q^{r/2 + \deg(v, F)/2 + o(r)}\sum_{\substack{\deg t < r }}q^{\deg(F,t)} \nonumber\\
        &\ll q^{3r/2 + \deg(v,F)/2 + o(r)}.
\end{align}

\textit{Case IV: $j \geq 2$, and $F\nmid u$ and $F\nmid v$}.  First we isolate $t=0$ and apply Lemma \ref{lem:ramanujan} and see 
\begin{align*}
   \left|T_{F,a}(u,v)\right|
   &\ll \Big{|}\sum_{\substack{\deg t < r \\ t \neq 0}}K_F(u,t)K_F(v,t)e_F(-at)\Big{|} + q^{\deg(u,F) + \deg(v,F)}.
\end{align*}
For a given $t \neq 0$ by Lemma \ref{lem:kloos=0_gcd_not_equal}, $K_F(u,t)K_F(v,t)$ vanishes unless either 
\begin{align}\label{eq:gcds_equal}
    \gcd(u,F) = \gcd(v,F) = \gcd(t,F)
\end{align}
or $P^{j-1}$ divides each of $\gcd(s,F), \gcd(t,F)$ and $ \gcd(v,F)$. But since $F$ divides neither of $\gcd(s,F), \gcd(t,F)$ nor $\gcd(v,F)$, both cases imply (\ref{eq:gcds_equal}). Thus if $g = \gcd(v,F) = \gcd(u,F)$ then by Lemma \ref{lem:kloos_divide_by_gcd}, 
\begin{align*}
  & \left|T_{F,a}(u,v)\right| \\
   &\quad\ll q^{2\deg g}\\
   & \qquad\qquad + \Big{|}q^{\deg g}\sum_{\substack{\deg t < r \\ t \neq 0 \\ g|t}}K_{F/g}(u/g,t/g)K_{F/g}(v/g,t/g)e_{F/g}(-at/g) \Big{|}\\
   &\quad =  q^{2\deg g} \\
   &\qquad\qquad +\Big{|}q^{\deg g}\sum_{\substack{\deg t < r-\deg g \\ t \neq 0}}K_{F/g}(u/g,t)K_{F/g}(v/g,t)e_{F/g}(-at) \Big{|}.
\end{align*}
We can now add and subtract $t=0$ and apply Lemma \ref{lem:ramanujan}, implying 
\begin{align}\label{eq:T_bound_caseIV_1}
   \left|T_{F,a}(u,v)\right|
   &\ll q^{2\deg g} +q^{\deg g}\left|T_{F_0, a}(u_0, v_0)\right|
\end{align}
where $F_0 = F/g$, $u_0 = u/g$ and $v_0 = v/g$ and we set $r_0 = \deg F - \deg g$. We now focus on bounding 
$$T_{F_0, a}(u_0, v_0) = \left|\sum_{\deg t < r_0}K_{F_0, a}(u_0, t)K_{F_0, a}(v_0, t)e_{F_0}(-at)\right|.$$
By construction, we know $F_0 \neq 1$. If $F_0 = P$ then the result follows from the first two cases, so we may assume that $F_0 = P^{j'}$ for some $j' \geq 2$. For a given $t$, if $P | t$ then Lemma \ref{lem:kloos=0_if_gcd=1_and_other_divisible} implies $K_{F_0}(u_0,t) = K_{F_0}(v_0,t) = 0$ since $P \nmid u_0$ and $P\nmid v_0$, and thus we only need to consider $t$ coprime with $F$. Lemma \ref{lem:lem:explicit_kloosterman_primepower} then implies 
$$K_{F_0}(u_0,t)K_{F_0}(v_0,t) = 0$$
unless there exists some $\ell, \ell' \in \F_q[T]$ such that $t \equiv \ell^2u_0 \Mod{F_0}$ and $t \equiv {\ell'}^2v_0 \Mod{F_0}$. Note this implies there exists some $c$ such that $u_0 \equiv c^2v_0 \Mod{F_0}$. Instead of summing over $t$ we instead sum over $\ell$, which implies
\begin{align*}
\left|T_{F_0,a}(u_0,v_0)\right|
&\ll \Big{|}\sum_{\substack{\deg \ell < r_0 \\ (\ell,F_0) = 1}}K_{F_0}(u_0, \ell^2u_0)K_{F_0}({c}^2{u_0}, \ell^2u_0)e_{F_0}(-au_0\ell^2)\Big{|}.
\end{align*}
Now we can apply Lemma \ref{lem:lem:explicit_kloosterman_primepower} to explicitly evaluate these Kloosterman sums as
\begin{align*}
|T_{F_0,a}&(u_0,v_0)|\\
&\ll \bigg{|}\sum_{\substack{\deg \ell < r_0 \\ (\ell,F_0) = 1}}\left(\legendre{\ell u_0}{F_0}_qq^{r_0/2}\textup{Re} ~\epsilon_{F_0}e_{F_0}(2\ell u_0)\right)\\
&\qquad\qquad \times \left(\legendre{\ov{c}\ell u_0}{F_0}_qq^{r_0/2} \textup{Re} ~\epsilon_{F_0}e_{F_0}(2\ell\ov{c}u_0)\right)e_{F_0}(-au_0\ell^2)\bigg{|}\\
&\ll q^{r_0}\Big{|}\sum_{\substack{\deg \ell < r_0\\ (\ell,F_0) = 1}}\left(\textup{Re}~\epsilon_{F_0}e_{F_0}(2\ell u_0)\right)\left(\textup{Re}~\epsilon_{F_0}e_{F_0}(2\ell\ov{c} u_0)\right)e_{F_0}(-au_0\ell^2)\Big{|}.
\end{align*}
Since $\epsilon_{F_0}$ lies on the unit circle and does not depend on $\ell$, after expanding real parts we can factor out any dependence on it, and we have 
\begin{align*}
%%&\ll  q^r\sum_{\substack{\deg \ell < r \\(\ell,F) = 1}}\legendre{\ell u}{F}_q\epsilon_F(e_F(2\ell u) + e_F(-2\ell u))\legendre{\ell \ov{c} u}{F}_q\epsilon_F(e_F(2\ell \ov{c} u) + e_F(-2\ell \ov{c} u))e_F(-au\ell^2)\\
%%&= q^r \sum_{\substack{\deg \ell < r \\ (\ell,F) = 1}}\left(e_F(-au\ell^2 + 2u\ell(1 + \ov{c})) + e_F(-au\ell^2 + 2u\ell(1 - \ov{c})) + e_F(-au\ell^2 - 2u\ell(1 - \ov{c})) +e_F(-au\ell^2 - 2u\ell(1 + \ov{c}))\right)\\
&\left|T_{F_0,a}(u_0,v_0)\right|\\
&\:\:\ll  q^{r_0} \Big{|} \sum_{(k_1,k_2) \in \{0,1\}^2}\sum_{\substack{\deg \ell < r_0 \\ (\ell,F_0) = 1}} e_{F_0}(-au_0\ell^2 + (-1)^{k_1}2u_0\ell(1 + (-1)^{k_2}\ov{c})) \Big{|}. 
\end{align*}
This sum is equal to four sums Gauss sums over reduced residue classes. Applying Lemma \ref{lem:gauss_sum_reduced_residues} now gives 
\begin{align*}
    \left|T_{F_0,a}(u_0,v_0)\right|
    &\ll q^{3r_0/2+\deg(au_0, u_0(1\pm \ov{c}), F_0)/2}.
\end{align*}
Now using $u_0 \equiv c^2v_0 \Mod{F}$ we have 
\begin{align*}
    \gcd(u_0 ~\pm ~\ov{c}u_0, F) &=\gcd(u_0 ~\pm ~cv_0, F_0)\\
    & \leq \gcd(u_0^2 - c^2v_0^2, F_0) \\
    &= \gcd(u_0^2-u_0v_0, F_0) \\
    &\leq  \gcd(u_0-v_0, F_0).
\end{align*}
Thus we conclude 
\begin{align*}
  \left|T_{F_0,a}(u_0,v_0)\right|
\ll q^{3r_0/2 + \deg(u_0-v_0, a, F_0)/2}.
\end{align*}
Now substituting into (\ref{eq:T_bound_caseIV_1}) we can finally conclude 
\begin{align}\label{eq:T_bound_caseIV}
   \left|T_{F,a}(u,v)\right|
   &\ll q^{2\deg g} +q^{3r/2 - \deg g/2 + \deg(u/g-v/g, a, F/g)/2} \nonumber \\
   &\ll q^{3r/2 +\deg(u,v,F)/2 + \deg(u-v, a, F/g)/2}.
\end{align}

Putting \textit{Case I, Case II, Case III} and \textit{Case IV} together by combining (\ref{eq:T_bound_caseI}), (\ref{eq:T_bound_caseII}), (\ref{eq:T_bound_caseIII}) and (\ref{eq:T_bound_caseIV}) now produces the result. 
\end{proof}

\newpage
\section{Notation Guide}\label{notationguide}
\begin{center}
    \begin{longtable}{ p{3cm} p{8cm}  }
    \hline 
    \multicolumn{2}{l}{General Notation} \\
  $q$ & a prime power: if $q$ is required to be odd, then this is specified \\
  $\F_q$ & the finite field of order $q$ \\
   $\F_q[T]$ &  the ring of univariate polynomials with coefficients from $\F_q$\\
   $\F_q(T)$ & the field of fractions of $\F_q[T]$\\
   $\Ki$ & the field of Laurent series in $1/T$ over $\F_q$, or equivalently the completion of $\F_q(T)$ at infinity\\
    $F$     &     a polynomial in $\F_q[T]$ \\
    $r$     &      the degree of $F$     \\
    \hline
    \multicolumn{2}{l}{Functions} \\
       $|~\cdot~|$ & the absolute value on $\Ki$ at infinity\\
    $e(~\cdot~)$ & the canonical additive character of  $\Ki$\\
    $e_F(~\cdot~)$ & the canonical additive character of $\F_q[T]/\langle F\rangle $\\
        $\legendre{~\cdot~}{~\cdot~}_q$ & the Legendre-Jacobi symbol in $\F_q[T]$\\
        $\mu_q(~\cdot~)$    & the M{\"o}bius function in $\F_q[T]$           \\
            $\deg_Fx$ & for $x \in \F_q[T]$, the degree of the unique polynomial $x'$ satisfying $\deg x' < r$ and $x' \equiv x \Mod{F}$\\
    $\ov{x}$ & for $x \in \F_q[T]$, the multiplicative inverse of $x$ modulo $F$ (if the inverse taken to a different modulus, this is specified)\\
     \hline
    \multicolumn{2}{l}{Sets and Weights} \\
   $\cS, \cT$     &      arbitrary finite subsets of $\F_q[T]$, often only containing elements of degree less than $r$    \\
    $\cI_m$        &     the set $\{s + s_0 : s \in \F_q[T], \deg                       s < m\}$ for some $s_0 \in \F_q[T]$  \\
     $\cI_n$       &     the set $\{t + t_0 : t \in \F_q[T], \deg                       t < m\}$ for some $t_0 \in \F_q[T]$  \\
     $\balpha$, $\bbeta$ & sequences of complex weights on some finite sets in $\F_q[T]$\\
     \\
     \hline
     \multicolumn{2}{l}{Exponential Sums} \\
     $K_F(~\cdot~)$ & a Kloosterman sum in $\F_q[T]/\langle F \rangle $\\
     $\cK_{F,\cdot}(~\cdot~)$ & a bilinear form of Kloosterman sums in $\F_q[T]/\langle F \rangle$ \\
     $G_F(~\cdot~)$ & a Gauss sum in $\F_q[T]/\langle F \rangle$\\
     $G^*_F(~\cdot~)$ & a Gauss sum over reduced residue classes in $\F_q[T]/\langle F \rangle$\\
     $\cG_{F,\cdot}(~\cdot~)$ & a bilinear form of Gauss sums in $\F_q[T]/\langle F \rangle$ \\
     $C_F(~\cdot~)$ & a Ramanujan sum in $\F_q[T]/\langle F \rangle$\\
      \hline
    \multicolumn{2}{p{11cm}}{Counting Functions} \\
    \multicolumn{2}{p{11cm}}{Note: for all of the functions below if $\cI_m$ (resp. $\cI_n$) is replaced by just the integer $m$ (resp. $n$), this indicates that $\cI_m$ (resp. $\cI_n$) is an initial interval} \\
       $H_{F,a}(\cI_m, \cI_n)$     &  the number of solutions to
       \begin{center}
           $x_1x_2 \equiv a \Mod{F}$
       \end{center} with $x_1 \in \cI_m$ and $x_2 \in \cI_n$ \\
   $I_{F,a}(\cI_m)$ & the number of solutions to 
   \begin{center}
       $\ov{x}_1 + \ov{x}_2 \equiv a \Mod{F}$
   \end{center} with $x_1, x_2 \in \cI_m$    \\
   $A_{F,a}(\cI_m, k)$ & the number of solutions to 
   \begin{center}
       $\ov{x}_1 + \ov{x}_2 \equiv ax_3 \Mod{F}$
   \end{center}
   with $x_1, x_2 \in \cI_m$ and $\deg x_3 < k$   \\
   $E_F^\inv(\cI_m)$ & the number of solutions to 
   \begin{center}
       $\ov{x_1} + \ov{x_2} \equiv \ov{x_3} + \ov{x_4} \Mod{F}$
   \end{center}
with $x_i \in \cI_m$  \\
   $E_F^\sqrt(m)$ & the number of solutions to 
   \begin{center}
       ${x_1} + {x_2} \equiv {x_3} + {x_4} \Mod{F}$
   \end{center}
with $\deg_F(x_i^2) < m$ and $\deg x_i < r$  \\
   $E_F^\sq(\cI_m)$ & the number of solutions to 
   \begin{center}
       $x_1^2 + {x_2}^2 \equiv {x_3}^2 + {x_4}^2 \Mod{F}$
   \end{center}
with $x_i \in \cI_m$  \\
 \end{longtable}
\end{center}

\end{document}